\documentclass{amsart}

\usepackage{amsmath,amsthm,amssymb}
\usepackage{mathtools}
\usepackage{tikz,tikz-3dplot,tikz-cd,graphicx}
\usepgflibrary{shadings}

\graphicspath{{./}{figures/}}

\usetikzlibrary{arrows,shapes.geometric}

\theoremstyle{plain}
\newtheorem{thm}{Theorem}[section]
\newtheorem{lem}[thm]{Lemma}
\newtheorem{cor}[thm]{Corollary}
\newtheorem{prop}[thm]{Proposition}

\newtheorem{thma}{Theorem}

\theoremstyle{definition}
\newtheorem{defn}[thm]{Definition}
\newtheorem{rem}[thm]{Remark}

\newtheorem{example}[thm]{Example}

\newcommand{\N}{\mathbb{N}}

\newcommand{\R}{\mathbb{R}}

\newcommand{\C}{\mathbb{C}}
\newcommand{\D}{\mathbb{D}}

\newcommand{\A}{\mathbb{A}}\newcommand{\Ab}{\mathbf{A}}
\newcommand{\B}{\mathbb{B}}\newcommand{\Bb}{\mathbf{B}}
\newcommand{\I}{\mathbb{I}}\newcommand{\Ib}{\mathbf{I}}
\renewcommand{\S}{\mathbb{S}}\newcommand{\Sb}{\mathbf{S}}

\newcommand{\pb}{\mathbf{p}}\newcommand{\bb}{\mathbf{b}}

\newcommand{\inter}{\textrm{int}}

\newcommand{\roots}{\textsc{Roots}}
\newcommand{\critpts}{\textsc{CritPts}}
\newcommand{\critvals}{\textsc{CritVls}}
\newcommand{\cat}{\textsc{CAT}}
\newcommand{\rts}{\mathbf{rts}}
\newcommand{\cpt}{\mathbf{cpt}}
\newcommand{\cvl}{\mathbf{cvl}}

\newcommand{\hgt}{\mathbf{hgt}}
\renewcommand{\arg}{\mathbf{arg}}
\newcommand{\zer}{\mathbf{0}}
\newcommand{\lat}{\textsc{Lat}}
\newcommand{\lng}{\textsc{Long}}
\newcommand{\lvl}{\textsc{Lvl}}
\newcommand{\dir}{\textsc{Dir}}
\newcommand{\NC}{\textsc{NC}}

\newcommand{\multi}{\textsc{Mult}}
\newcommand{\sym}{\textsc{Sym}}

\newcommand{\poly}{\textsc{Poly}}
\newcommand{\size}[1]{|#1|}

\newcommand{\vb}{\mathbf{v}}

\newcounter{joncomments}

\newcounter{michaelcomments}

\usetikzlibrary{shapes.geometric}
\usetikzlibrary{calc}
\usetikzlibrary{decorations.markings}
\usetikzlibrary{arrows.meta}

\tikzstyle{BlueLine}=[line width=0.3mm,color=blue,text=black]
\tikzstyle{BluePoly}=[BlueLine,fill=blue!20]
\tikzstyle{RedLine}=[line width=0.3mm,color=red,text=black]
\tikzstyle{RedPoly}=[RedLine,fill=red!20]
\tikzstyle{GreenLine}=[thick,color=black!30!green,text=black]
\tikzstyle{GreenPoly}=[thick,color=green!50!black,fill=green!30,join=bevel]
\tikzstyle{PurpleLine}=[line width=0.3mm,color=red!60!blue!80,text=black]
\tikzstyle{PurplePoly}=[PurpleLine,fill=red!40!blue!50]
\tikzstyle{OrangeLine}=[thick,color=orange]
\tikzstyle{GrayLine}=[thick,color=black!50!gray]
\tikzstyle{GrayPoly}=[GrayLine,fill=gray!20]
\tikzstyle{dot}=[shape=circle,draw,color=black,fill=black,inner sep=1.5pt]
\tikzstyle{opendot}=[shape=circle,draw,color=black,fill=white,inner sep=1.5pt]
\tikzstyle{bigdot}=[dot,inner sep=2pt]
\tikzstyle{littledot}=[dot,inner sep=0.75pt]
\tikzstyle{littleopendot}=[opendot,inner sep=0.75pt]
\tikzstyle{smalldot}=[dot,inner sep=1pt]
\tikzstyle{disk}=[thick,shape=circle,draw,color=black,fill=yellow!10]
\tikzstyle{lightdisk}=[thick,shape=circle,draw,color=black!40,fill=yellow!20]
\tikzstyle{plate}=[thick,shape=rectangle,draw,color=black,fill=yellow!10,rounded corners,minimum size=1.1cm]

\def\shade{.8}

\definecolor{c00}{rgb}{0,\shade,\shade} 
\definecolor{c60}{rgb}{0,0,\shade}
\definecolor{c120}{rgb}{\shade,0,\shade} 
\definecolor{c180}{rgb}{\shade,0,0}
\definecolor{c240}{rgb}{\shade,\shade,0} 
\definecolor{c300}{rgb}{0,\shade,0}
\definecolor{c360}{rgb}{0,\shade,\shade}

\colorlet{c10}{c60!16.7!c00} \colorlet{c20}{c60!33.3!c00}
\colorlet{c30}{c60!50!c00}   \colorlet{c40}{c60!66.7!c00}
\colorlet{c50}{c60!83.3!c00}

\colorlet{c70}{c120!16.7!c60} \colorlet{c80}{c120!33.3!c60}
\colorlet{c90}{c120!50!c60}   \colorlet{c100}{c120!66.7!c60}
\colorlet{c110}{c120!83.3!c60}

\colorlet{c130}{c180!16.7!c120} \colorlet{c140}{c180!33.3!c120}
\colorlet{c150}{c180!50!c120}   \colorlet{c160}{c180!66.7!c120}
\colorlet{c170}{c180!83.3!c120}

\colorlet{c190}{c240!16.7!c180} \colorlet{c200}{c240!33.3!c180}
\colorlet{c210}{c240!50!c180}   \colorlet{c220}{c240!66.7!c180}
\colorlet{c230}{c240!83.3!c180}

\colorlet{c250}{c300!16.7!c240} \colorlet{c260}{c300!33.3!c240}
\colorlet{c270}{c300!50!c240}   \colorlet{c280}{c300!66.7!c240}
\colorlet{c290}{c300!83.3!c240}

\colorlet{c310}{c360!16.7!c300} \colorlet{c320}{c360!33.3!c300}
\colorlet{c330}{c360!50!c300}   \colorlet{c340}{c360!66.7!c300}
\colorlet{c350}{c360!83.3!c300}

\colorlet{ca}{c00}
\colorlet{cb}{c60}
\colorlet{cc}{c180}
\colorlet{cd}{c300}

\colorlet{bckgrd}{green!10!white}
\colorlet{cA}{black!10!white}
\colorlet{cB}{black!20!white}
\colorlet{cC}{black!40!white}
\colorlet{cD}{black!60!white}
\colorlet{cE}{black!80!white}

\def\anga{0} \def\angb{60} \def\angc{180} \def\angd{-60}

\def\rp{.02}   

\begin{document}

\title[Geometric Combinatorics of Polynomials I]{Geometric
  Combinatorics of Polynomials I:\\ the case of a single polynomial}

\author{Michael Dougherty}
\author{Jon McCammond}
\date{August 1, 2020}

\begin{abstract}
  There are many different algebraic, geometric and combinatorial
  objects that one can attach to a complex polynomial with distinct
  roots. In this article we introduce a new object that encodes many
  of the existing objects that have previously appeared in the
  literature.  Concretely, for every complex polynomial $p$ with
  $d$ distinct roots and degree at least $2$, we produce a canonical
  compact planar $2$-complex that is a compact metric version of a
  tiled phase diagram.  It has a locally $\cat(0)$ metric that is
  locally Euclidean away from a finite set of interior points indexed
  by the critical points of $p$, and each of its $2$-cells is a metric
  rectangle. From this planar rectangular $2$-complex one can 
  use metric graphs known as metric cacti and metric banyans
  to read off several pieces of combinatorial data:
  a chain in the partition lattice, a cyclic factorization of a
  $d$-cycle, a real noncrossing partition (also known as a primitive
  $d$-major), and the monodromy permutations for the polynomial.  
  This article is the first in a series.
\end{abstract}

\maketitle

\begin{center}
\textit{Dedicated to the memory of Patrick Dehornoy.}
\end{center}

\section*{Introduction}

Complex polynomials have long been central objects in many fields, and
efforts to understand them have incorporated several different algebraic, 
geometric and combinatorial tools.  Examples include the study of
polynomial lemniscates \cite{catanese-paluszny,catanese-wajnryb,
  cacti-braids-polynomials,zvonkin00}, monodromy groups
\cite{catanese-paluszny,catanese-wajnryb,humphries03,michel06,wegert20}, 
metric cacti \cite{nekrashevych14}, and modulus graphs
\cite{epstein-hanin-lundberg}.  In this article we add a new item to
this list.  Recall that if $p\colon\C\to\C$ is a complex polynomial
with $p'(b) = 0$, then $b$ is a \emph{critical point} and $p(b)$ is
the corresponding \emph{critical value}.  It is well known that if $p$
has distinct roots, then the critical points are disjoint from the
roots and thus the critical values of $p$ are nonzero.  If we write
$\C_\rts$ and $\C_\zer$ to denote the complex plane with the roots
removed or with zero removed, respectively, then $p$ may be restricted
to a map $p_\zer \colon\C_\rts\to\C_\zer$.  Our main theorems concern
the manipulation of this restricted polynomial map into a map between
metric cell structures on compact surfaces.

By a straightforward projection, $\C_\zer$ may be identified with the
interior of a compact annulus $\A$. If we pull the standard annular
metric back to $\C_\rts$ via $p_\zer$, then $\C_\rts$ becomes a
bounded metric space which is an open ``branched'' annulus, and its
metric completion $\B_p$ is a compact metric surface with
boundary. The result is a new discretely branched map $p_A:\B_p \to\A$
between compact metric surfaces with boundary.

\begin{thma}[Geometry]\label{thma:metric-planar-complex}
  Let $p$ be a complex polynomial with $d$ distinct roots.  Then
  $\B_p$ is a compact metric surface of genus $0$ with $d+1$ boundary
  components. The metric for $\B$ is locally $\cat(0)$ and is locally
  Euclidean away from a finite set of interior points indexed by the
  critical points of $p$.
\end{thma}

If $p$ has degree at least $2$, then it has critical values and these
determine a rectangular cell structure for $\A$ which we refer to as
$\Ab_p$. Pulling back via $p_A$ provides the metric rectangular cell
complex $\Bb_p$ which can be seen as a compact metric version of a
tiled phase diagram (as described in \cite{wegert20}) and makes $p$
into a cellular map $\pb\colon\Bb_p\to\Ab_p$.

\begin{thma}[Topology]\label{thma:topology}
  Let $p$ be a complex polynomial with $d$ distinct roots and degree
  at least $2$.  Each vertex on the boundary of the metric cell
  complex $\Bb_p$ has degree $3$. Every other vertex is labeled by a
  point in $\C$; those which are critical points of multiplicity $k$ 
  have valence $4(k+1)$, while the rest have valence $4$.
\end{thma}

A natural pair of transverse foliations on $\A$ emerges by envisioning
the annulus as being oriented vertically. One foliation consists of
horizontal latitude circles and the other consists of vertical
longitude lines. In each case, we refer to a foliation leaf as
\emph{critical} if it contains a critical value of $p$ and
\emph{regular} if it does not.  The pullbacks of these foliations in
$\B_p$ determine the cell structure of $\Bb_p$ in the sense that each
critical leaf is a subcomplex of $\Ab_p$ and its preimage is a
subcomplex of $\Bb_p$. For simplicity, we refer to the preimage of a
latitude as a \emph{level set} and the preimage of a longitude as a
\emph{direction set}.  Each of these preimages (or equivalently
certain subcomplexes of $\Bb_p$) determines combinatorial objects
associated to $p$.

\begin{thma}[Combinatorics]\label{thma:combinatorics}
  Let $p$ be a complex polynomial with $d$ distinct roots.  Each
  regular level set determines a partition of the set $\rts$ of
  roots. Taken together, the collection of all regular level sets
  determines a chain in the partition lattice $\Pi_\rts$. Meanwhile,
  each regular direction set determines a cyclic ordering of the set
  of roots.  Taken together, the collection of all regular direction
  sets determines a factorization of the $d$-cycle
  $(1\ \cdots\ d)$. Finally, given any regular level set and regular
  direction set, the partition determined by the former is noncrossing
  with respect to the permutation determined by the latter.
\end{thma}

In the discussion surrounding Theorem~\ref{thma:combinatorics}, we
also describe several useful combinatorial objects associated to a
complex polynomial. These include metric graphs known as \emph{metric
  cacti} and \emph{metric banyans}, as well as a continuous variant of
a noncrossing partition which we refer to as a \emph{real noncrossing
  partition}. Some of these tools are very closely related to the
\emph{primitive $d$-major} defined in \cite{thurston19}.

This article is the first in a series in which we highlight the strong
connections between complex polynomials with distinct roots and a
variety of metric and combinatorial objects.  Future articles will use
these tools to examine the space of polynomials and address how these
objects change as polynomials are continuously varied.  As an example,
we will show that the combinatorial data outlined in
Theorem~\ref{thma:combinatorics} can be used to reconstruct the cell
structure for $\Bb_p$; in particular, two polynomials are
topologically equivalent if they have the same associated chain of
partitions and factorization of the $d$-cycle $(1\ \cdots\ d)$.

More generally, the goal of this series of articles is to describe a
natural stratification (and compactification) of the space of complex
polynomials with $d$ distinct ordered roots (also known as the
complement of the complex braid arrangement) into strata with locally
flat Euclidean metrics. We will also describe a deformation retraction
of this space onto a subcomplex which is isometric to the \emph{pure
  dual braid complex}, a simplicial complex associated to the braid
group \cite{brady01,bradymccammond10}.  This embedding of the pure
dual braid complex in the complex braid arrangement complement is new
to the literature, but the existence of such an embedding has been
known to Daan Krammer for some time.  This deformation retraction is
also closely related to the retraction described in \cite{thurston19}.

The present article is structured as
follows. Sections~\ref{sec:disks-annuli} and~\ref{sec:draw-a}
establish some conventions regarding disks, annuli and how to draw
them. Section~\ref{sec:poly-roots} reviews some basic information on
roots, critical points and critical values, and we use this in
Section~\ref{sec:metric-cells} to construct the branched annulus and
establish its key properties. In Section~\ref{sec:draw-b}, we give a
detailed description of how to draw a planar picture for the branched
annulus.  Section~\ref{sec:level-direction} detail the combinatorial
information encoded in the two foliations of this surface. Finally,
Sections~\ref{sec:partitions},~\ref{sec:factorizations}
and~\ref{sec:monodromy} describe how a chain of root partitions, a
factorization of a $d$-cycle and the monodromy action can be read off
of the branched annulus.

\vspace{1em}

\noindent\textbf{Acknowledgements:} This article is dedicated to the
memory of Patrick Dehornoy.  In the short film that he created as part
of the celebration of his retirement, Patrick depicted himself looking
down from paradise and keeping track of the afterlife of his
mathematical work.  Patrick's ideas have had an enormous impact on the
direction of our research, and as this series of articles progresses,
we hope to show that the Garside structures he pioneered are deeply
connected to the geometric combinatorics of complex polynomials.  The
authors are also deeply indebted to Daan Krammer; his comments to the
second author (at Patrick's retirement conference) prompted this line
of inquiry.  Finally, we would also like to thank Steve Trettel for
his help in visualizing the foliations described in
Section~\ref{sec:level-direction}.

\section{Disks and Annuli}\label{sec:disks-annuli}

This section describes two elementary homeomorphisms: one from the
complex numbers to the interior of a disk using polar coordinates, and
another from the nonzero complex numbers to the interior of an annulus
using cylindrical coordinates.  The map to the annulus is used to
define the branched annulus that is our main object of study.  The map
to the disk is only used to give a planar representation of
the resulting surface.  For simplicity, let $\zer = \{0\}$ and let
$\C_A$ be $\C \setminus A$ for any $A \subset \C$.  In particular,
$\C_\zer = \C^\ast$.

\begin{defn}[Polar and cylindrical coordinates]
  Every nonzero complex number has a unique \emph{polar form} $z =
  ru$, where $r = \size{z} \in \R^+$ is its \emph{magnitude} and $u =
  z/r \in \S^1 \subset \C$ is its \emph{argument} or \emph{direction}.
  In particular, the map $z \mapsto (r,u)$ identifies $\C_\zer$ with
  $\R^+ \times \S^1$.  At the origin $r=0$ and $u$ is arbitrary.  If
  we identify $\C$ with the $xy$-plane in $\mathbb{R}^3$, via the
  pairing of $x+iy$ with $(x,y,0)$, then polar coordinates $(r,u)$ on
  $\C$ can be extended to cylindrical coordinates $(r,u,t)$ on $\R^3$.
  On the $t$-axis, $r=0$ and $u$ is arbitrary. Note that \emph{height}
  in $\R^3$ is denoted $t$, since $z$'s are used for complex numbers.
\end{defn}

\begin{defn}[Unit disk]\label{def:unit-disk}
  Let $\D \subset \C$ be the closed unit disk and let $\D^\inter$ be
  its interior.  Let $i_D\colon\C\hookrightarrow\D$ be the map that
  sends $z=ru$ to $i_D(z) = su$ where $s = r/\sqrt{r^2+1}$.  The map
  $i_D$ is a homeomorphism from $\C$ to $\D^\inter$.  Geometrically,
  it sends the horizontal coordinates of the hyperboloid model of the
  hyperbolic plane to the Klein model.  For clarity, we use $s$ for
  the magnitude of a complex number in $\D$ and $r$ for the magnitude
  of a complex number in $\C$.
\end{defn}

The second homeomorphism is constructed using stereographic
projection.

\begin{defn}[Riemann sphere]\label{def:riemann-sphere}
  The $2$-sphere in $\R^3$ can be viewed as the one-point
  compactification of the complex plane via stereographic projection.
  Projection from the point $(0,0,1)$ sends the point $z = ru \in \C$
  to the point with cylindrical coordinates
  $(\frac{2r}{1+r^2},u,\frac{r^2-1}{r^2+1}) \in \S^2$.  This sends
  $\C$ to the $2$-sphere minus the north pole and the origin is sent
  to the the south pole.  Thus the image of $\C_\zer$ is the
  twice-punctured $2$-sphere with both the north and south poles
  removed.
\end{defn}

\begin{defn}[Vertical annulus]\label{def:annulus}
  Let $\S = \S^1 \subset \C$, let $\I = [-1,1] \subset \R$, and let
  $\A = \S \times \I$ be the closed \emph{vertical annulus} formed by
  the points in $\R^3$ with cylindrical coordinates $(1,u,t)$, $u \in
  \S$ and $t \in \I$.  These are the points that are distance $1$ from
  the $t$-axis, in an arbitrary direction in $\S$ and with a height in
  $\I$.  The interior of $\A$ is the open vertical annulus $\A^\inter
  = \S \times \I^\inter$. Let $i_A\colon \C_\zer \to \A$ be the map
  that sends $z=ru$ to the point $i_A(z) = (1,u,t)$, where
  $t=(r^2-1)/(r^2+1)$. Then $i_A$ is a homeomorphism from $\C_\zer$ to
  $\A^\inter$.  Geometrically, it is the result of stereographically
  projecting the once-punctured plane to the twice-punctured sphere
  and then radially projecting away from the $t$-axis to the vertical
  annulus $\A$.
\end{defn}

\begin{figure}
  \includegraphics[width=4in]{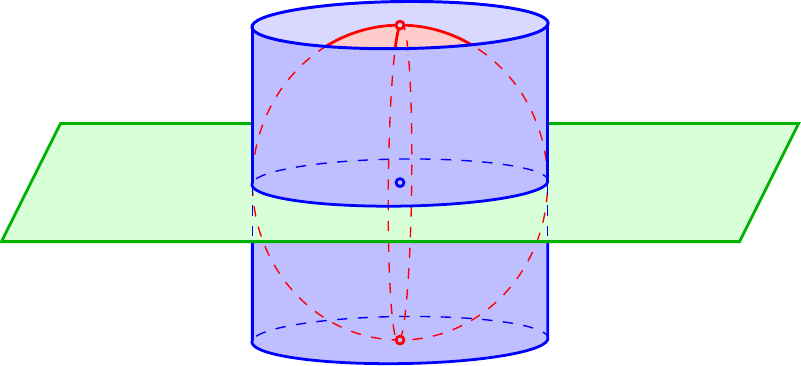}
  \caption{The once-punctured plane, the twice-punctured sphere and
    the vertical annulus.}
  \label{fig:three-spaces-together}
\end{figure}

Figure~\ref{fig:three-spaces-together} shows the geometric
relationship between the once-punctured plane, the twice-punctured
sphere and the vertical annulus.  The vertical annulus is a metric
product and there are projection maps onto each factor.

\begin{defn}[Projection maps]
  Let $\hgt \colon \A \to \I$ be the map that sends $(1,u,t)$ to its
  height $t$ and let $\arg \colon \A \to \S$ be the map that sends
  $(1,u,t)$ to its argument $u$.  We call $\hgt(x) \in \I$ the
  \emph{height of $x$} and we call $\arg(x) \in \S$ the
  \emph{argument of $x$}.
\end{defn}

Latitude circles and longitude lines are standard subsets of the
twice-punctured sphere, and we borrow these names for the
corresponding subsets of the annulus.

\begin{defn}[Latitude circles]\label{def:lat}
  A latitude circle on the $2$-sphere radially projects to a
  horizontal circle in $\A$.  We call $\lat_t = \hgt^{-1}(t) =
  \{(1,u,t) \mid u \in \S\}$ the \emph{latitude circle at height $t$}
  in $\A$. For each $t\in \I^\inter$, the preimage of $\lat_t \subset
  \A^\inter$ is the circle $i_A^{-1}(\lat_t) = \{ru \mid u \in\S^1\}
  \subset \C_\zer$ of radius $r= \sqrt{(1+t)/(1-t)}$ and the image
  $i_D(i_A^{-1}(\lat_t)) \subset \D$ is a circle of radius $s =
  \sqrt{(1+t)/2}$.  Both circles are centered at the origin.  The
  latitude circles $\lat_1$ and $\lat_{-1}$ are the \emph{boundary
    circles} of $\A$, and note that $\partial \A = \A \setminus
  \A^\inter = \lat_1 \cup \lat_{-1}$.
\end{defn}

\begin{defn}[Longitude lines]\label{def:long}
  A longitude line on the $2$-sphere radially projects to a vertical
  segment in $\A$.  We call $\lng_u = \arg^{-1}(u) = \{(1,u,t) \mid t
  \in \I\}$ the \emph{longitude line in direction $u$} in $\A$.  Note
  that it is the closure in $\A$ of the radial projection.  For each
  $u \in \S$, the preimage of the longitude line $\lng_u \subset \A$
  is the open ray $i_A^{-1}(\lng_u) = \{ru \mid r \in\R^+\} \subset
  \C_\zer$ and the image $i_D(i_A^{-1}(\lng_u)) \subset \D$ is the
  open line segment from $0$ to $u$ in $\D$.
\end{defn}

\begin{rem}[Transverse measures]
  The latitude circles and longitude lines form a very simple pair of
  transverse measured foliations arising from the Euclidean structure
  on the vertical annulus $\A$. The interval $\I$, which indexes the
  latitude circles, has measure $2$, the circle $\S$, which indexes
  the longitude lines, has measure $2\pi$, and the annulus $\A$ has
  area $4\pi$.
\end{rem}

The annulus $\A$ can be given a $2$-complex structure built out of
rectangles.

\begin{defn}[Points and rectangles]\label{def:rectangles}
  If $U \subset \S$ is a non-empty finite set of directions of size
  $\size{U} = k$, then there is a natural minimal cell structure on
  $\S$ that contains the $k$ elements of $U$ as vertices and the $k$
  intervals between them as edges.  We write $\Sb$ instead of $\S$
  when the circle has been given a specific cell structure.  If $T
  \subset \I^\inter$ is a finite set of heights of size
  $\size{T}=\ell$, then there is a natural minimal cell structure on
  $\I = [-1,1]$ that contains the $\ell+2$ elements of $T \cup
  \{-1,1\}$ as vertices and the $\ell+1$ intervals between them as
  edges.  We write $\Ib$ instead of $\I$ when this interval has been
  given a specific cell structure.  Given $U$ and $T$ as described
  above, let $\Ab = \Sb \times \Ib$ be the natural product cell
  structure on $\A$.  Note that this is the minimal cell structure on
  $\A$ that contains the latitude circles $\{\lat_t\}_{t\in T}$ and
  the longitude lines $\{\lng_u\}_{u\in U}$ as subcomplexes.  The
  $\ell$ latitude circles decompose the vertical annulus $\A$ into
  $\ell+1$ shorter annuli and the $k$ longitude lines decompose each
  annulus into $k$ rectangles.  Thus $\Ab$ has $k(\ell+2)$ $0$-cells,
  $k(2\ell+3)$ $1$-cells and $k(\ell+1)$ $2$-cells, and each $2$-cell
  is a metric rectangle.  More generally, one can create such a cell
  structure $\Ab$ from any non-empty finite subset $V \subset \A^\inter$.
  Simply define $U = \arg(V)$ and $T = \hgt(V)$ and proceed as above.
  The resulting rectangular $2$-complex $\Ab$ is the smallest
  rectangular tiling of this type that contains all of $V$ in its
  $0$-skeleton.
\end{defn}

\begin{example}[Four points]\label{ex:four-points}
  Consider the four points $\{.46,-1.62,.3\pm .56i\} \subset \C_\zer$
  and let $V \subset \A^\inter$ be the image of this set under $i_A$.
  The points in $\C_\zer$ have three magnitudes and four arguments, so
  their images in $\A$ have three heights and four arguments. The
  circle $\Sb$ has $4$ vertices and $4$ edges, the interval $\Ib$ has
  $5$ vertices and $4$ edges, and the $2$-complex $\Ab = \Sb \times
  \Ib$ has $20$ vertices, $36$ edges and $16$ rectangles.
\end{example}

For later use we introduce a notation for the connected covers of $\S$
or $\Sb$.

\begin{defn}[Covers of a circle]\label{def:circle-covers}
  For each positive integer $k$, let $\S(k) \subset \C$ denote the
  circle of radius $k$ centered at the origin.  The map $z\mapsto
  (z/k)^k$ restricts to a $k$-fold covering and a local isometry
  $\S(k) \to \S$. If $\S$ is given a cell structure, $\Sb$, then we
  can lift through the cover to obtain a cell structure for $\S(k)$
  which we denote as $\Sb(k)$, resulting in a cellular covering map
  $\Sb(k) \to \Sb$.
\end{defn}

\section{Drawing an Annulus}\label{sec:draw-a}

It is sometimes convenient to display structures on the vertical
annulus $\A \subset \R^3$ by embedding $\A$ into the unit disk $\D$.
This distorts the intrinsic metric of $\A$, but it produces a planar
image.  The composition $i_D \circ i_A^{-1}$ almost works for this
purpose since it maps the open annulus $\A^\inter$ homeomorphically
to $\D^\inter_\zer$, the open disk with the origin removed.
However, the extension of this map obtained by metrically completing
both domain and range is no longer a homeomorphism.
The upper boundary $\lat_1$ in $\A$ is sent homeomorphically to 
$\partial \D$, but at that bottom of the annulus, the entire circle 
$\lat_{-1}$ is sent to the origin.  To correct for this, we make a 
small modification to pull the image of the open annulus away from 
the puncture.

\begin{defn}[Enlarging a puncture]\label{def:enlarge-punct}
  Let $D = \D(\alpha)$ be the closed disk of radius $\alpha>0$
  centered at the origin and let $D_\zer = D \setminus \zer$.  The map
  $z \mapsto \alpha(z/\size{z})$ retracts the punctured disk $D_\zer$
  onto its boundary $\partial D_\zer = \alpha \S$, the circle of
  radius $\alpha$, and the straightline homotopy $H_\epsilon\colon D_0
  \to D_0$ sending $z$ to $z + \epsilon(\alpha(z/\size{z})-z)$, with
  $\epsilon \in [0,1]$, shows that this retraction is a deformation
  retraction.  Moreover, $H_\epsilon$ is a homeomorphism onto its
  image for all $0 \leq \epsilon < 1$, creating a homeomorphsim
  between $D_\zer$ and $D_B = D \setminus B$, where $B = \D(\beta)$ is
  the closed disk of radius $\beta = \epsilon \cdot \alpha < \alpha$.
  This can be extended to the rest of $\C$ by fixing all points
  outside $D$, and the extended $H_\epsilon \colon \C_\zer \to
  \C_\zer$ creates a homeomorphism between $\C_\zer$ and its image
  $\C_B$ that is the identity outside of $\C_D$.
\end{defn}

\begin{figure}
\begin{tikzcd}
  \C_\zer \arrow[r,"H_\epsilon"] 
  \arrow[d,hookrightarrow,"i_A"] &
  \C_\zer \arrow[r,hookrightarrow,"i"] & 
  \C \arrow[d,hookrightarrow,"i_D"] \\
  \A \arrow[rr,hookrightarrow,"j_{AD}"] & &
            \D  \\
\end{tikzcd}
\caption{Drawing the annulus $\A$ inside the disk
  $\D$.\label{fig:drawing-A}} 
\end{figure}

\begin{defn}[Drawing $\A$ inside $\D$]\label{def:drawing-annuli}
  Pick small numbers $\alpha > \beta > 0$ and define the extended map
  $H_\epsilon \colon \C_0 \to \C_0$ as described in
  Definition~\ref{def:enlarge-punct} with $\epsilon = \beta/\alpha$.
  With this slight pertubation around the origin, the composition $i_D
  \circ H_\epsilon \circ i_A^{-1}$ maps the open annulus
  homemorphically to the open unit disk with a closed neighborhood of
  the origin removed, and this homeomorphism does extend to a
  well-behaved embedding $j_{AD} \colon \A \hookrightarrow \D$ between
  their metric completions.  See Figure~\ref{fig:drawing-A}.  In
  particular, $j_{AD}$ homeomorphically embeds $\A$ into $\D$ and it
  agrees with the metric completion of the map $i_A^{-1} \circ i_D$
  except in a small neighborhood of the lower boundary circle
  $\lat_{-1}$.  We use this type of identification whenever we draw
  $\A$ inside $\D$.
\end{defn}

\begin{example}[Four points, revisited]
  Example~\ref{ex:four-points} described a metric rectangular cell
  structure on the vertical annulus.  Figure~\ref{fig:rect-tiling}
  shows the corresponding tiling mapped to the disk $\D$ using the
  embedding $j_{AD}$ (Definition~\ref{def:drawing-annuli}). The sides
  of the rectangles appear curved in the figure, but they are true
  Euclidean rectangles since the metric comes from that of the
  vertical annulus $\A$.
\end{example}

\begin{figure}
\begin{tikzpicture}	
  \begin{scope}[scale=0.04]
    \def\rp{2.5} \def\rd{1.5}
    \def\radA{15} \def\radB{23} \def\radC{{sqrt(1009)}} 
    \def\radD{81} \def\radE{100} \def\radF{110}
    \coordinate (cva) at (\anga:\radB);
    \coordinate (cvb) at (\angb:\radC);
    \coordinate (cvc) at (\angc:\radD);
    \coordinate (cvd) at (\angd:\radC);
    \draw[fill=bckgrd] (0,0) circle (\radE);
    \draw[thick,black] (0,0) circle (\radE);
    \foreach \x/\y in {D/\radD,C/\radC,B/\radB}{\draw[thick,c\x] (0,0) circle (\y);}
    \draw[thick,cA,fill=white] (0,0) circle (\radA);
    \foreach \x/\y in {ca/\anga,cb/\angb,cc/\angc,cd/\angd}{
      \draw[thick,\x] (\y:\radA) to (\y:\radE);}
    \foreach \x in {a,b,c,d}{
      \draw [thick,color=c\x,fill=c\x!70] (cv\x) circle (\rp);
    }
  \end{scope}
  \end{tikzpicture}
  \caption{A metric rectangular cell complex on the annulus induced by
    four points and then drawn in the disk.  The images of the four
    original points are drawn as large dots.}\label{fig:rect-tiling}
\end{figure}
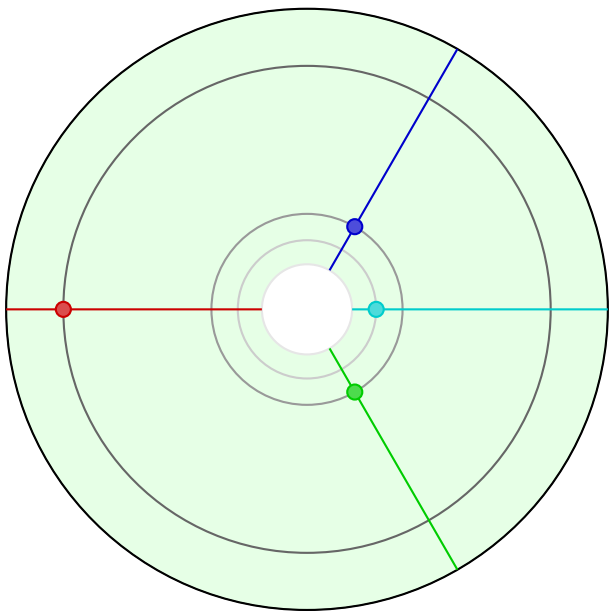

\section{Polynomials with Distinct Roots}\label{sec:poly-roots}

This section records basic facts about the roots, critical points and
critical values of a complex polynomial.

\begin{defn}[$3$ sets]\label{def:3-sets}
  Let $\poly_d(\C) \subset \C[z]$ denote the subset of all polynomials
  of degree $d$ in $\C[z]$ with $d \in \N$.  A \emph{complex
    polynomial of degree $d$} is conventionally written in additive
  form using $d+1$ coefficients: $p(z) = c_d z^d + \cdots c_1 z^1 +
  c_0$, where the \emph{leading coefficient} $c_d$ is nonzero and all
  of the others are arbitrary, but the behavior of the polynomial map
  $p\colon \C \to\C$, sending $z$ to $p(z)$, is easier to analyze when
  the formula is written multiplicatively.  By the Fundamental Theorem
  of Algebra $p(z)$ can be factored completely: $p(z) = c_d
  (z-a_1)(z-a_2)\cdots (z-a_d)$ with $c_d\in \C_\zer$ and $a_i \in \C$
  for all $i\in \{1,2,\ldots,d\}$.  The elements of the set
  $\{a_1,a_2,\ldots,a_d\} \subset \C$ are the \emph{roots of $p$} and
  the number $c_d$ is its \emph{leading coefficient}.  The
  \emph{critical points of $p$} are the roots of its derivative and
  the \emph{critical values of $p$} are the images of its critical
  points.  In symbols:
  \[
  \begin{array}{ccccl}
    \rts &=& \roots(p) &=& \{ a \in \C \mid p(a) = 0\}\\
    \cpt &=& \critpts(p) &=& \{ b\in \C \mid p'(b)=0 \}\\
    \cvl &=& \critvals(p) &=& \{c \in \C \mid c=p(b) , b \in \cpt\}\\
  \end{array} 
  \]
\end{defn}

As a mnemonic device, we tend to use $a$'s, $b$'s and $c$'s when
naming individual roots, critical points and critical values,
respectively.  Note that $\rts$ and $\cpt$ are subsets in the domain
and $\cvl$ is a subset in the range.  Moreover, $p(\rts) = \zer$,
$\rts = p^{-1}(\zer)$, $p(\cpt) = \cvl$ and $\cpt \subset
p^{-1}(\cvl)$.  The map $p$ restricts to a map $p_\zer \colon \C_\rts
\to \C_\zer$.  When the roots of $p$ are distinct, $\rts$ has size $d$
and the polynomial can be recovered from its set of roots and its
leading coefficient.  We say `set' rather than `list' since the
ordering of the linear factors is clearly irrelevant. An arbitrary
polynomial can be recovered when one records the multiplicity of each
root.

\begin{defn}[Multisets]\label{def:multisets}
  A \emph{multiset} is a set $S$ together with a \emph{multiplicity
    function} $m\colon S \to \N$.  The number $m(s)$ is the
  \emph{multiplicity of $s\in S$}.  The multiplicity function $m$ is
  usually left implicit and a multiset is named after its underlyng
  set.  A finite multiset can be concisely described using the
  notation $S = \{s_1^{m_1},s_2^{m_2},\ldots, s_k^{m_k}\}$ where the
  underlying set is $S = \{s_1,\ldots,s_k\}$, with $s_i = s_j$ if and
  only if $i=j$, and where the exponent $m_i = m(s_i)$ denotes the
  multiplicity of the element $s_i$.  The \emph{size} of the multiset
  $S$ is the sum of its multiplicities: $d = \sum_i m_i$, but the
  \emph{size} of the (underlying) set $S$ is simply $k$. Note that $d
  >k$ unless every element has multiplicity $1$.  The collection of
  all multisets of complex numbers of size~$d$ is denoted
  $\multi_d(\C)$.
\end{defn}

\begin{defn}[$3$ multisets]\label{def:3-multisets}
  Let $p \in \poly_d(\C)$ be a polynomial of degree $d$.  The roots,
  critical points and critical values of $p$ can be converted from
  sets to multisets by specifying their multiplicity functions.  The
  multiplicity of a root $a \in \rts$ is the number of times that
  $(z-a)$ occurs as a linear factor in the complete factorization of
  $p(z)$.  The multiplicity function of a critical point $b \in \cpt$
  is the number of times that $(z-b)$ occurs as a linear factor in the
  complete factorization of $p'(z)$.  Finally, the multiplicity of a
  critical value $c\in \cvl$ is the sum of the multiplicities of the
  critical points sent to $c$, i.e. $m(c) = \sum_{b\in \cpt_c} m(b)$,
  where $\cpt_c = \{b \in \cpt \mid p(b)=c\}$.
\end{defn}

For a generic polynomial all three of these multisets are sets.  The
polynomial in Example~\ref{ex:3-ms-generic} is our standard example
that is used throughout the article.

\begin{example}[$3$ sets: generic case]\label{ex:3-ms-generic} 
  For $p(z) = .02(3z^5 - 15z^4 + 20z^3 - 30z^2 + 45z)$, all three multisets
  are sets.  It has a set of five distinct roots $\rts =
  \{a_1,a_2,a_3,a_4,a_5\}$ that we have indexed at random.  Three of
  the roots are real ($a_1 =0$, $a_2 \approx 1.7944$ and $a_3 \approx
  3.5972$), and there is one complex conjugate pair ($a_4 = x + yi$
  and $a_5 = x- yi$ where $x \approx -0.1958$ and $y \approx 1.5117$).
  The critical points and critical values are easier to describe
  because the derivative factors as $p'(z) = .3(z-1)(z-3)(z^2+1)$.  It
  has a set of four distinct critical points $\cpt =\{1, 3, \pm i\}$
  and a set of four distinct critical values $\cvl = \{.46, -1.62, .3 \pm
  .56i\}$.  The $3$ multisets for this polynomial are shown in
  Figure~\ref{fig:3-ms-generic}.  Note that $\cvl$ is the set used in
  Example~\ref{ex:four-points}.
\end{example}

\begin{figure}
    \begin{tikzpicture}[scale=.8]
    \begin{scope}[xshift=-5cm,scale=1]
    \draw (-1.5,-2.5) rectangle (5,2.5);
    \draw[->] (-1,0)--(4.5,0);
    \draw[->] (0,-2)--(0,2);
    \coordinate (a0) at (0,0);
    \coordinate (a1) at (1.7944,0);
    \coordinate (a2) at (3.5972,0);
    \coordinate (a3) at (-0.1958,1.5117);
    \coordinate (a4) at (-0.1958,-1.5117);
    \foreach \x in {0,1,2,3,4}{\draw [ultra thick, fill=white] (a\x) circle [radius=0.1];}
                    
    \coordinate (b1) at (1,0);
    \coordinate (b2) at (3,0);
    \coordinate (b3) at (0,1);
    \coordinate (b4) at (0,-1);
    \foreach \x in {1,2,3,4} {\draw [ultra thick,fill=blue] (b\x) circle [radius=0.1];}
   
    \end{scope}
    
    \begin{scope}[xshift=5cm,scale=.04]
    \draw (-110,-60) rectangle (60,60);
    \draw[->] (-100,0)--(50,0);
    \draw[->] (0,-50)--(0,50);
    \coordinate (c0) at (23,0);
    \coordinate (c1) at (-81,0);
    \coordinate (c2) at (15,28);
    \coordinate (c3) at (15,-28);
    \foreach \x in {0,1,2,3}{\draw [ultra thick, fill=red] (c\x) circle [radius=2.5];}
    \end{scope}
    
    \end{tikzpicture}
    \caption{Three finite sets associated with the polynomial given in
      Example~\ref{ex:3-ms-generic}.  The domain, on the left,
      contains its five roots, shown as white dots, and its four
      critical points, shown as blue dots.  The range, on the right,
      contains its four critical values, shown as red dots.  The two
      images use different scales.}
    \label{fig:3-ms-generic}
\end{figure}
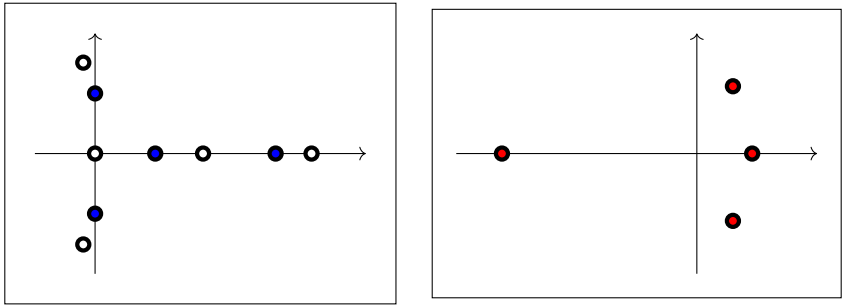

At the other extreme, there might be just one critical point and one
critical~value.

\begin{example}[$3$ multisets: special case]\label{ex:3-ms-special} 
  Let $p(z) = a(z-b)^d + c$ with $a \in \C_\zer$ and $b,c \in \C$.  If
  $n=d-1$, then the derivative is $p'(z) = a\cdot d (z-b)^n$.  The
  multiset $\cpt = \{b^n\}$; there is only one critical point of
  multiplicity $n$.  And since $p(b) = c$, the multiset $\cvl =
  \{c^n\}$; there is only one critical value, also of multiplicity
  $n$.  The roots of $p$ are complex numbers of the form $b$ plus a
  $d$-th root of $-c/a$, so the character of the roots depends on the
  value of $c$.  If $c=0$, then $\rts = \{b^d\}$; there is only one
  root of multiplicity $d = n+1$.  But if $c \neq 0$, then there $d$
  distinct roots equally spaced around a circle centered at $b$.
\end{example}

In fact, a polynomial with only one critical value must be one of
those listed in Example~\ref{ex:3-ms-special}.  We record some basic
facts about polynomials for later use.

\begin{prop}[Polynomials and roots]\label{prop:poly-rts}
  If $p$ is a complex polynomial of degree $d$, then its multiset of
  roots has size $d$.  Conversely, every polynomial $p$ can be
  uniquely reconstructed from its leading coefficient and its multiset
  of roots.  In particular, there is a natural bijection $\poly_d(\C)
  \cong \C_\zer \times \multi_d(\C)$.
\end{prop}

\begin{proof}
  Given $(c_d,\{a_1^{m_1},\ldots,a_k^{m_k}\})$ in $\C_\zer \times
  \multi_d(\C)$ with $\sum_i m_i = d$, simply define $p(z) = c_d
  (z-a_1)^{m_1}\cdots(z-a_k)^{m_k}$.
\end{proof}

\begin{prop}[Polynomials and derivatives]\label{prop:poly-deriv}
  Every polynomial is determined by its derivative and its value at
  one point.  Explicitly, when $p \in \C[z]$ be a polynomial with
  $p(b) =c$, $p(z) = \int_b^z p'(w) dw + c$ and, in the special case
  where $a$ is a root, this simplifies to $p(z) = \int_a^z p'(w) dw$.
\end{prop}

\begin{proof}
  For any $b \in \C$, the integral $\int_b^z p'(w) dw +c = p(z) - p(b)
  + c$, which is equal to $p(z)$ when $p(b)=c$.
\end{proof}

\begin{prop}[Polynomials and critical points]\label{prop:poly-cpt}
  Every polynomial can be reconstructed from its leading coefficient,
  its multiset of critical points and its value at one point.
\end{prop}

\begin{proof}
  The leading coefficient of $p$ and the leading coefficient of $p'$
  are closely related, in the sense that each determines the other 
  and the critical points of $p$ are the roots of $p'$.  Thus, $p'$ 
  can be reconstructed from the leading coefficient of $p$ and the 
  multiset of critical points of $p$ 
  (Proposition~\ref{prop:poly-rts}), and then $p$ can be 
  reconstructed from $p'$ and the value of $p$ at one point 
  (Proposition~\ref{prop:poly-deriv}).
\end{proof}

\begin{rem}[Polynomials and critical values]\label{rem:poly-cvl}
  Every multiset of size $n$ of complex numbers can be realized as the
  critical values of a polynomial of degree $d = n+1$, but the proof
  is somewhat indirect, and the map from polynomials to critical
  values is finite-to-one, adding to the ambiguity \cite{beardon02}.
  In particular, it is fairly difficult to explicitly construct one of
  the finitely many polynomials have a specified multiset as its
  multiset of critical values.  
\end{rem}

Polynomials with distinct roots have many characterizations.

\begin{prop}[Distinct roots]\label{prop:distinct-roots}
  For a polynomial $p \in \poly_d(\C)$ the following are equivalent:
  $(1)$ $p$ has $d$ distinct roots, $(2)$ the roots and critical
  points of $p$ are disjoint multisets, $(3)$ $0$ is not a critical
  value of $p$, i.e. $\cvl \subset \C_\zer$.
\end{prop}

\begin{proof}
  If $a \in \rts$ is a root, then $p(a)=0$, and $p(z) = (z-a) q(z)$
  for some polynomial $q(z)$.  In particular, $p'(a) = \lim_{z\to a}
  \frac{p(z)-p(a)}{z-a} = \lim_{z\to a} q(z) = q(a)$.  Thus $q(a)=0$
  if and only if $p'(a)=0$, which means that $p$ has a multiple root
  at $a$ if and only if $a$ is both a root and a critical point of
  $p$.  Thus $(1)$ and $(2)$ are equivalent.  When there is an element
  that is both a root and a critical point, $\rts \cap \cpt$ is not
  empty and $p(\rts \cap \cpt) \subset p(\rts) \cap p(\cpt) = \zer
  \cap \cvl$ is also not empty, so not $(2)$ implies not $(3)$.
  Conversely, when $0$ is an element of $\cvl$, there is an element
  $b\in \cpt$ such $p(b)=0$, so $b$ is also a root, and not $(3)$
  implies not $(2)$.
\end{proof}

When critical values exist, they create a cell structure on the
annulus.

\begin{defn}[Critical values and rectangles]\label{def:crit-val-cplx}
    Let $p$ be a complex polynomial and let $\cvl$ be its set of
    critical values. If $p$ has distinct roots and degree at least
    $2$, then $\cvl$ is a non-empty finite subset of $\C_\zer$
    (Proposition~\ref{prop:distinct-roots}). As described in
    Definition~\ref{def:annulus}, this subset of $\C_\zer$ is sent by
    $i_A$ to a non-empty subset of $\A^\inter$.  By
    Definition~\ref{def:rectangles}, this subset determines cell
    structures for $\I$, $\S$ and $\A = \S \times I$, which we denote
    $\Ib_p$, $\Sb_p$ and $\Ab_p$, respectively.
\end{defn}

\section{Polynomials and Branched Annuli}\label{sec:metric-cells}

This section defines our main object of study: the branched 
annulus of a polynomial with distinct roots. 
We begin by recalling the topology of polynomial maps.

\begin{prop}[Polynomial maps]\label{prop:poly-maps}
  For every polynomial $p \in \poly_d(\C)$ , the map $p\colon \C \to
  \C$ is a local homeomorphism away from $\cpt$, and every point not
  in $\cvl$ has an evenly covered neighborhood.  In particular, if $B
  \subset \C$ is any set containing $\cvl$, then $p$ restricts to a
  $d$-sheeted covering $\C_C \to \C_B$, where $C = p^{-1}(B)$.
\end{prop}

In other words, polynomial maps are branched covering maps 
with finitely many branch points. It is helpful to study
this branching behavior using the graph of the modulus function.

\begin{defn}[Modulus surface]\label{def:modulus-surface}
  Let $p \in \poly_d(\C)$ be a polynomial and let $p\colon \C \to \C$
  be the corresponding polynomial map.  The \emph{modulus surface of
    $p$} is the graph of the function $\size{p} \colon \C \to \R$
  which sends $z \mapsto \size{p(z)}$.  This is a surface in $\C
  \times \R$, formed by the points $\{(z,\size{p(z)}) \mid z\in \C\}$,
  and it is smooth except possibly near the roots.  Near a root with
  multiplicity $1$, the modulus surface looks like a cone. The roots
  live in the plane $\C \times \zer$ and the remainder of the modulus
  surface is in the half space above this plane. The graph of the
  restricted function $\size{p_\zer}\colon \C_\rts \to \R$, is the
  \emph{restricted modulus surface} and its level sets are called
  \emph{lemniscates}.
\end{defn}

\begin{rem}[Compactifying the modulus surface]\label{rem:cpt-modulus-surface}
  One inconvenience when studying the restricted modulus surface is
  that both the domain and range are non-compact.  This can be fixed
  by homeomorphically mapping $\C_\rts$ into a compact disk $\D$,
  mapping $\R$ into an compact interval $\I$, and then metrically
  completing the resulting bounded surface contained inside the solid
  cylinder $\D \times \I$.
\end{rem}

The branched annulus introduced here is an alternative way to
compactify the restricted modulus surface. The first step is to define a
bounded metric on $\C_\rts$.

 \begin{defn}[Pullback metrics]\label{def:pullback-metrics}
   If $Y$ is a topological surface, $X$ is a Riemannian surface, and
   $f\colon Y \to X$ is a branched cover with a finite branch locus,
   then there is a unique \emph{pullback metric} on $Y$ that makes $f$
   a local isometry wherever it is a local homeomorphism.  Concretely,
   call a curve in $Y$ \emph{rectifiable} if its image under $f$ is
   rectifiable in $X$, define the length of a rectifiable curve in $Y$
   by the length of its image in $X$ and define the distance between
   points in $Y$ to be the infimum of the lengths of rectifiable
   curves connecting them.  This defines a metric on $Y$ that makes
   $f$ a local isometry in any neighborhood where $f$ was already a
   local homeomorphism, while also providing equal treatment to all
   points in $Y$.  Its uniqueness follows from the local isometry
   requirement.
 \end{defn}

Note that pullback metrics are distance non-increasing as an
immediate consequence of their definition.

\begin{defn}[Branched annulus]\label{def:brannulus-metric}
  Let $p \in \poly_d(\C)$ be a polynomial with distinct roots.  The
  \emph{open branched annulus of the polynomial $p$} is the
  topological space $\C_\rts$ endowed with a pullback metric via the
  map $i_A \circ p_\zer \colon \C_\rts \to \A$, and the (closed)
  \emph{branched annulus} is the metric completion of the open
  branched annulus.  We write $\B_p^\inter$ for the open branched
  annulus, $\B_p$ for the branched annulus, and $i_B \colon \C_\rts
  \to \B_p$ for the natural inclusion map which restricts to a
  homeomorphism between $\C_\rts$ and $\B_p^\inter$.  Finally, since
  the map $i_A \circ p_\zer \circ i_B^{-1} \colon \B_p^\inter \to \A$
  is distance non-increasing, it sends Cauchy sequences to Cauchy
  sequences, and thus it continuously extends to a map $p_A \colon
  \B_p \to \A$. See Figure~\ref{fig:poly-maps}.
  \end{defn}
  
\begin{figure}
\begin{tikzcd}
\C  \arrow[d,"p"] \arrow[r,hookleftarrow] &
 \C_\rts \arrow[d, "p_\zer"]  \arrow[r,hookrightarrow,"i_B"] & 
  \B_p \arrow[r,equal] \arrow[d,"p_A"] & 
  \Bb_p \arrow[d,"\pb"] \\
  \C \arrow[r,hookleftarrow] &
  \C_\zer \arrow[r,hookrightarrow,"i_A"] &  
  \A \arrow[r,equal] &
    \Ab_p 
\end{tikzcd}
\caption{Maps related to the complex polynomial $p$.  The four
  vertical maps, from left to right, are the original polynomial map
  $p$, which is a branched cover of $\C$, the restricted polynomial
  map $p_0$, which is a branched cover of $\C_\zer$, the compact
  metric version $p_A$, which is a branched cover of the annulus $\A$,
  and the metric cellular version $\pb$, which is a branched
  metric cellular map between rectangular
  $2$-complexes. \label{fig:poly-maps}}
\end{figure}

It turns out that the branched annulus $\B_p$ is always a compact
surface with boundary whose interior is the open branched annulus
$\B_p^\inter$, hence our choice of notation.  These properties are
established later in the section.  For now we merely record the
elementary properties of the open branched annulus $\B_p^\inter$.
  
\begin{prop}[Open branched annulus]\label{prop:open-brannulus}
  For any polynomial $p \in \poly_d(\C)$ with distinct roots, the open
  branched annulus $\B_p^\inter$ is a connected genus $0$ surface that
  is locally Euclidean away from a finite set of critical points. It
  has area $4\pi d$ and its diameter is bounded.
\end{prop}

\begin{proof}
  The open branched annulus $\B_p^\inter$ is a connected genus $0$
  surface because its topology is the same as $\C_\rts$.  It is
  locally Euclidean almost everywhere since it is a branched cover of
  $\A$ with finite branch locus and $\A$ is locally Euclidean.  It has
  total area $4\pi d$ since the map is a local isometry and a
  $d$-sheeted covering of a space with area $4 \pi$, once finitely
  many points have been removed from the domain and range
  (Propostion~\ref{prop:poly-maps}).  And finally, a connected
  finite-sheeted cover of a bounded diameter metric space, has bounded
  diameter.
\end{proof}
 
To understand the metric completion of an open branched annulus, note
that a Cauchy sequence with missing limit point in the open branched
annulus $\B_p^\inter$ must approach the ``boundary'' of the underlying
topological space $\C_\rts$.  In other words, such a sequence either
approaches a root or it heads off to infinity (in the old metric).
The image of such a sequence under the map $p_A$ approaches
$\partial \A = \lat_{-1} \cup \lat_1$.

\begin{rem}[Near a root]\label{rem:near-a-root}
  Let $p\in \poly_d(\C)$ be a polynomial with distinct roots.  Since
  $0$ is not a critical value (Proposition~\ref{prop:distinct-roots}),
  there is a small neighborhood $C$ of $0$ that is evenly covered by
  $p$ (Proposition~\ref{prop:poly-maps}).  Let $a \in \rts$ be a root
  and let $B$ be the connected component of $p^{-1}(C)$ containing
  $a$.  Since the map $i_A \circ p_\zer$ sends the deleted
  neighborhood $B \setminus \{a\}$ homeomorphically to the lower
  portion of the open annulus $\A^\inter$, the metric completion of
  $\B_p^\inter$ in the neighborhood of the root $a$ involves adding a
  copy of the boundary circle $\lat_{-1}$, as happens when metrically
  completing the lower part of $\A^\inter$ to get $\A$.  In
  particular, $d$ distinct circles of length $2\pi$ are added to the
  open branched annulus, one for each of the $d$ roots in $\rts$, and
  in each case, the portion of $\B_p^\inter$ near this boundary circle
  is the interior of $\B_p$ near this circle.  We call these boundary
  circles \emph{root circles}.
\end{rem}

\begin{rem}[Near infinity]\label{rem:near-infinity}
  Let $p\in \poly_d(\C)$ be a polynomial with distinct roots and note
  that the points in $\C$ that are ``near infinity'' are sent by $i_A
  \circ p_\zer$ to points that are near the boundary circle $\lat_1
  \in \A$.  If $A$ is an open neighborhood of $\lat_1 \in \A$ that is
  disjoint from $i_A(\cvl)$, then the restriction of $i_A \circ
  p_\zer$ to the preimage of $A \cap \A^\inter$ is a $d$-sheeted cover
  of this annular region (Proposition~\ref{prop:poly-maps}).  It is
  also path-connected since it contain all points in $\C$ sufficiently
  far from the origin.  The only possibility is that $(i_A \circ
  p_\zer)^{-1}(A \cap \A^\inter)$ is an open topological annulus that
  is sent to the open annulus $A \cap \A^\inter$ by a map that wraps
  around $d$ times.  As a consequence, a single circle $\S(d)$, of
  length $2 \pi d$, is added to $\B_p$ ``near infinity'' and the
  portion of $\B_p^\inter$ near this boundary circle is the interior
  of $\B_p$ near this circle.  We call this the \emph{circle at
    infinity}, since it is reminiscent of the standard
  compactification of the complex plane.  Note that this agrees with
  the intuition that polynomials near infinity are increasingly
  dominated by their leading term.
\end{rem}

The following proposition records the consequences of these remarks.

\begin{prop}[Branched annulus]\label{prop:brannulus}
  For any polynomial $p \in \poly_d(\C)$ with distinct roots, the
  branched annulus $\B_p$ is a compact connected genus $0$ surface
  with boundary and the map $p_A \colon \B_p \to \A$ is a branched
  cover with a finite set of branch points in its interior.  The
  interior of the surface $\B_p$ is the open branched annulus
  $\B_p^\inter$ and its boundary consists of $d+1$ circles.  There are
  exactly $d$ circles of length $2\pi$ that are all sent to lower
  boundary circle $\lat_{-1}$ in $\A$ and one circle $\S(d)$ of length
  $2\pi d$ that is sent to the upper boundary circle $\lat_1$ in $\A$.
\end{prop}

This nearly completes the proof of
Theorem~\ref{thma:metric-planar-complex}.  Since locally Euclidean
implies locally $\cat(0)$, it only remains to show that the branched
annulus is also locally $\cat(0)$ in the neighborhood of a branch
point.  This is straightforward to show directly and even more clear
once we introduce a cell structure.

\begin{defn}[Cell structure]\label{def:cellular-branched-annulus}
  Let $p \in \poly_d(\C)$ be a complex polynomial with distinct roots
  and let $p_A \colon \B_p \to \A$ be corresponding branched cover of
  the annulus $\A$.  If the degree $d$ of $p$ is at least $2$, then
  $p$ has at least one critical value and $i_A(\cvl)$ is not empty.
  Recall that $\Ab_p$ denotes the closed metric annulus $\A = \S
  \times \I$ with rectangular cell structure $\Ab_p = \Sb_p \times
  \Ib_p$ determined by the non-empty set $i_A(\cvl)$
  (Definition~\ref{def:crit-val-cplx}).  The open cells of this cell
  structure partition $\Ab_p$ and the components of their preimages
  under the map $p_A$ partition $\B_p$.  In fact, these components
  endow $\B_p$ with its own cell structure.  To see this, note that
  the new $0$-skeleton is simply the full preimage of the $0$-skeleton
  of $\Ab_p$.  And since the branch points of the map are sent to
  vertices of $\Ab_p$, by the definition of its cell structure, every
  open $1$-cell or open $2$-cell in $\Ab_p$ is evenly covered by $d$
  disjoint subsets of $\B_p$ that are exact copies of this open
  $1$-cell or open $2$-cell.  Finally, the topology of the branched
  cover ensures that the attaching maps behave as expected.  We write
  $\Bb_p$ to denote the branched annulus $\B_p$ together with this
  ``pulled back'' metric cell structure derived from the map to
  $\Ab_p$, and we write $\pb \colon \Bb_p \to \Ab_p$ for the metric
  cellular map from $\Bb_p$ to $\Ab_p$.  See
  Figure~\ref{fig:poly-maps}.
\end{defn}

\begin{rem}[Boundary of $\Bb_p$]\label{rem:Bb-boundary}
  Recall from Definition~\ref{def:crit-val-cplx} that
  the two boundary circles of $\Ab_p$ are isomorphic copies 
  of $\Sb_p$.  In $\Bb_p$, the $d$ root circles 
  are copies of $\Sb_p$ and the circle at infinity
  is a copy of $\Sb_p(d)$ as defined in Definition~\ref{def:circle-covers}.
\end{rem}

In $\Bb_p$, the $2$-cells are metric rectangles, the $1$-cells are
metric line segments, and each $1$-cell is incident to exactly two
$2$-cells, so all of the interesting structure occurs near
the $0$-cells.  We call a vertex in $\Bb_p$ \emph{critical} if
it is in $i_B(\cvl)$ and \emph{regular} if not.

\begin{lem}[Regular vertices]\label{lem:reg-vert}
  A regular vertex of $\Bb_p$ has valence $3$ if it is 
  in the boundary and valence $4$ if it is in the interior.
\end{lem}

\begin{proof}
  When $\vb$ is a regular vertex in $\Bb_p$, $\pb$ is a local
  homeomophism near $\vb$ by Proposition~\ref{prop:poly-maps}, so the
  valence at $\vb$ is equal to the valence of its image $\pb(\vb) \in
  \Ab_p$.  The values listed are those for the rectangular cell
  structure on $\Ab_p$.
\end{proof}

\begin{lem}[Critical vertices]\label{lem:crit-vert}
  If $b \in \C_\rts$ is a critical point for $p$ with multiplicity
  $k$, then its image $\bb = i_B(b)$ is a critical vertex in $\Bb_p$
  of valence $4(k+1)$.
\end{lem}

\begin{proof}
  Since $b$ is a critical point of multiplicity $k$, $p'(z) = (z-b)^k
  q(z)$ for some polynomial $q(z)$ with $q(b) \neq 0$. By
  Proposition~\ref{prop:poly-deriv}, $p(z) = p(b) + \int_b^z (w-b)^k
  q(w) dw$.  When $z \approx b$, $q(z) \approx q(b)$ and the integral
  is approximately $q(b)\int_b^z (w-b)^k\ dw$. Thus, $p(z) \approx
  p(b) + \frac{q(b)}{k+1}(z-b)^{k+1}$.  In particular, in a small
  neighborhood of the $b$, $p$ is, roughly speaking, a shifted and
  rescaled version of $z \mapsto z^{k+1}$.  Because $\pb(\bb)$ is in
  the interior of $\Ab_p$, it has valence $4$ and this means that
  $\bb$ has valence $4(k+1)$.
\end{proof}

Since the $2$-cells in $\Bb_p$ are metric rectangles, the valence
information in Lemmas~\ref{lem:reg-vert}
and~\ref{lem:crit-vert} immediately implies that $\Bb_p$, and
the underlying metric space $\B_p$, are locally $\cat(0)$.  This
proves the following proposition and
completes the proofs of Theorems~\ref{thma:metric-planar-complex}
and~\ref{thma:topology}.

\begin{prop}[Locally $\cat(0)$]
  For every polynomial $p$ with distinct roots and degree at least
  $2$, the compact surface $\B_p$ is locally $\cat(0)$.
\end{prop}

\section{Drawing a Branched Annulus}\label{sec:draw-b}

The process used to draw the vertical annulus $\A$ inside the closed
unit disk $\D$ can be slightly modified to draw the branched annulus 
$\B_p$ inside the same disk.

\begin{figure}
\begin{tikzcd}
  \C_\rts \arrow[r,"\widetilde H_\epsilon"] 
  \arrow[d,hookrightarrow,"\mathbf{i_B}"] &
  \C_\rts \arrow[r,hookrightarrow,"i"] & 
  \C \arrow[d,hookrightarrow,"i_D"] \\
  \B \arrow[rr,hookrightarrow,"j_{BD}"] & &
            \D  \\
\end{tikzcd}
\caption{Drawing the branched annulus $\B$ inside the disk $\D$.\label{fig:drawing-B}}
\end{figure}

\begin{defn}[Drawing $\B_p$ inside $\D$]\label{def:drawing-branched-annuli}
  Pick small numbers $\alpha > \beta > 0$ and define the extended map
  $H_\epsilon \colon \C_0 \to \C_0$ as described in
  Definition~\ref{def:enlarge-punct} with $\epsilon = \beta/\alpha$.
  In particular, choose $\alpha$ small enough so that the disk of
  radius $\alpha$ centered at the origin is evenly covered by the
  polynomial $p$.  Then the slight perturbation $H_\epsilon$ inside
  the deleted neighborhood of the origin can be simultaneously lifted
  to similar slight perturbations around each of the deleted root
  neighborhoods in $\C_\rts$.  Let $\widetilde H_\epsilon\colon
  \C_\rts \to \C_\rts$ be this map where $p_\zer \circ \widetilde
  H_\epsilon = H_\epsilon \circ p_\zer$.  The composition $i_D \circ
  \widetilde H_\epsilon \circ i_B^{-1}$ maps the interior of the
  branched annulus homeomorphically to the open unit disk with $d$
  small closed disks removed, and this homeomorphism extends to a
  well-behaved embedding $j_{BD} \colon \B_p \hookrightarrow \D$
  between their metric completions.  See Figure~\ref{fig:drawing-B}.
  In particular, $j_{BD}$ homeomorphically embeds $\B_p$ into $\D$ and
  it agrees with the metric completion of the map $i_B^{-1} \circ i_D$
  except in small neighborhoods of the root circles.  We use this type
  of identification whenever we draw $\B_p$ inside $\D$.
\end{defn}

\begin{example}
  Let $p(z) = .02(3z^5 - 15z^4 + 20z^3 - 30z^2 + 45z)$ as in
  Example~\ref{ex:3-ms-generic}. The planar drawings of $\Bb_p$ and
  $\Ab_p$ are shown in Figure~\ref{fig:metric-completion}. The
  critical longitudes and latitudes in $\Ab_p$ are drawn in color and
  greyscale respectively, and their preimages in $\Bb_p$ are shaded
  accordingly.
\end{example}

\begin{figure}
    \centering
    \includegraphics[scale=0.7]{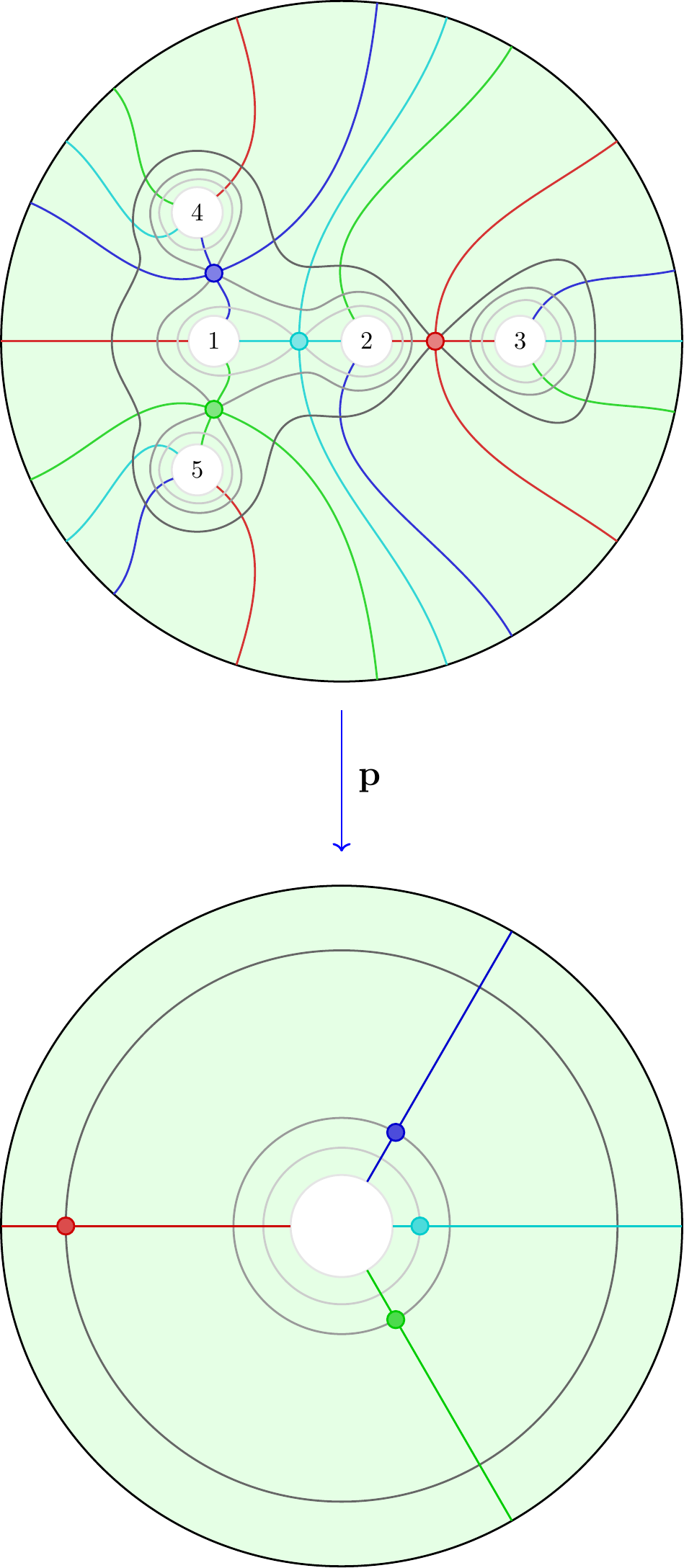}
    \caption{The metric cellular map $\pb$ from the metric rectangular
      $2$-complex $\Bb_p$ to the compact annulus $\Ab_p$ with its
      metric rectangular cell structure.  Both the domain and the
      range are embedded in the unit disk, even though this process
      distorts their intrinsic metrics.\label{fig:metric-completion}}
\end{figure}

\begin{rem}[Multipedal pants]\label{rem:mulitpedal-pants}
  The drawing of the branched annulus $\B_p$ in the disk $\D$ can be
  lifted to a surface in the solid cylinder $\D \times \I$ by adding
  in the height of a point under the map $\B_p \to \A \to \I$ as a
  third coordinate.  The root circles lie in the disk at the bottom of
  the cylinder, the circle at infinity is boundary of the disk at the
  top, and the level sets are the horizontal cross-sections.  For
  quadratic polynomials, this surface will be the ``pair of pants''
  familiar to topologists and in general, a polynomial of degree $d$
  will produce a pair of \emph{multipedal pants} designed for a being
  with $d$ legs.  An example of a $3$-legged pair of pants for a cubic
  polynomial is shown in Figure~\ref{fig:lvl-dir}.
\end{rem}

\section{Level Sets and Direction Sets}\label{sec:level-direction}

In the remainder of the article we turn our attention to the large
amount of information encoded in the metric rectangular complex
$\Bb_p$.  Latitudes and longitudes provide a pair of transverse
measured foliations for $\A$, and their preimages under $p_A$ provide
a pair of (singular) transverse measured foliations for $\B_p$ with
singularities at the critical points of $p$.  The leaves of each
pullback foliation can be written as metric graphs via the cell
structure for $\Bb_p$ and these provide valuable combinatorial data
associated to the polynomial $p$.  This section explores the geometry
and topology of these preimages.

\begin{defn}[Level sets and direction sets]\label{def:level-direction}
  By composing the polynomial $p_A \colon \B_p \to \A$ with the
  projections $\hgt\colon\A\to\I$ and $\arg\colon\A\to\S$, we obtain
  maps $\B_p \to \I$ and $\B_p \to \S$. For each $t\in \I$ or $u\in
  \S$, define the \emph{level set of height $t$} by $\lvl_t =
  p_A^{-1}(\lat_t) = \left(\hgt\circ p_A\right)^{-1}(t)$ and the
  \emph{direction set of argument $u$} by $\dir_u = p_A^{-1}(\lng_u) =
  \left(\arg\circ p_A\right)^{-1}(u)$. Since $p_A$ is a proper map,
  each of these preimages is a compact subset of $\B_p$, and note that
  the restricted maps $\hgt\circ p_A \colon \lvl_t\to\S$ and
  $\arg\circ p_A \colon \dir_u\to\I$ are discretely branched covers.
  We say that $\lvl_t$ (or $\dir_u$) is \emph{critical} when its image
  is a critical latitude (or longitude), and \emph{regular} otherwise.  
  Each latitude circle and
  longitude line can be made into a metric graph by intersecting it
  with the cell structure for $\Ab_p$, and the results are isomorphic
  to $\Sb_p$ and $\Ib_p$ respectively. In the same way, each level set
  $\lvl_t$ and direction set $\dir_u$ inherits a cell structure from
  $\Bb_p$, turning each of them into a metric graph and making the
  restricted maps $p_A\colon\lvl_t \to \lat_t$ and
  $p_A\colon\dir_u\to\lng_u$ into metric cellular maps.  Note,
  however, that the cell structures for $\lvl_t$ and $\dir_u$ are
  subcomplexes of $\Bb_p$ if and only if they are critical.
\end{defn}

While level sets are topologically equivalent to the usual definition
for polynomial lemniscates, the metric is different.

\begin{rem}[Lemniscates]\label{rem:erdos}
  For a polynomial $p$, the subset $\{z\in C \mid |p(z)| = r\}$ in
  $\C$ is known as a \emph{polynomial lemniscate} and its image under
  the inclusion map $i_B$ is the level set described above. While
  these two objects are homeomorphic, it is important to remember that
  the bounded metric on $\B_p$ differs from the standard metric on
  $\C$. For example, Erd\H{o}s, Herzog, and Piranian asked whether the
  lemniscate $\{z\in \C \mid |p(z)| = 1\}$ has maximum length when $p$
  the polynomial $p(z) = z^d-1$, and this problem has remained open
  for over sixty years \cite{erdos58}.  On the other hand,
  \emph{every} level set in $\B_p$ has total length $2\pi d$.
\end{rem}

It is easy to provide a precise geometric description of the regular
level sets and regular direction sets, but the critical versions
require a bit more exposition.  We begin with an elementary example.

\begin{example}[One critical value]\label{ex:lvl-dir}
  Let $p(z) = a(z-b)^d + c$ with $a,c \in \C_\zer$ and $b\in \C$. As
  described in Example~\ref{ex:3-ms-special}, $b$ is the unique
  critical point for $p$, $c$ is the corresponding unique critical
  value, and $p$ has distinct roots since $c$ is nonzero. Since there
  is only one critical point, $\Sb_p$ has $1$ vertex and $1$ edge, while
  $\Ib_p$ has $3$ vertices and $2$ edges.  Let $t$ and $u$ denote the
  height and argument of $i_A(c)$ in $\A$, respectively.
  Figure~\ref{fig:lvl-dir} shows the multipedal pants version of
  $\Bb_p$ when $d=3$, along with its unique critical level set
  $\lvl_t$ and unique critical direction set $\dir_u$.  Each level set
  has total length $2\pi d$ and there are two regular isometry types:
  $\lvl_s$ is the disjoint union of $d$ copies of $\Sb_p$ when $s < t$
  and $\lvl_s$ is isometric to $\Sb_p(d)$ when $s > t$.  The unique
  critical level set for $p$ occurs at height $t$ and is isometric to
  the $d$-fold wedge sum of $d$ copies of $\Sb_p$.  Meanwhile, each
  regular direction set is isometric to the disjoint union of $d$
  copies of $\Ib_p$ and the unique critical direction set may be
  written as the wedge sum of $d$ copies of $\Ib_p$. This critical
  direction set occurs at argument $u$ and has the structure of a
  metric tree with $2d$ leaves, each of which is connected via an edge
  to a single non-leaf vertex. Half of these edges have length $t$
  while the rest have length $2-t$, and the two types alternate in the
  local cyclic ordering around the non-leaf vertex.
\end{example}

\begin{figure}
  \centering
  \includegraphics[scale=0.3]{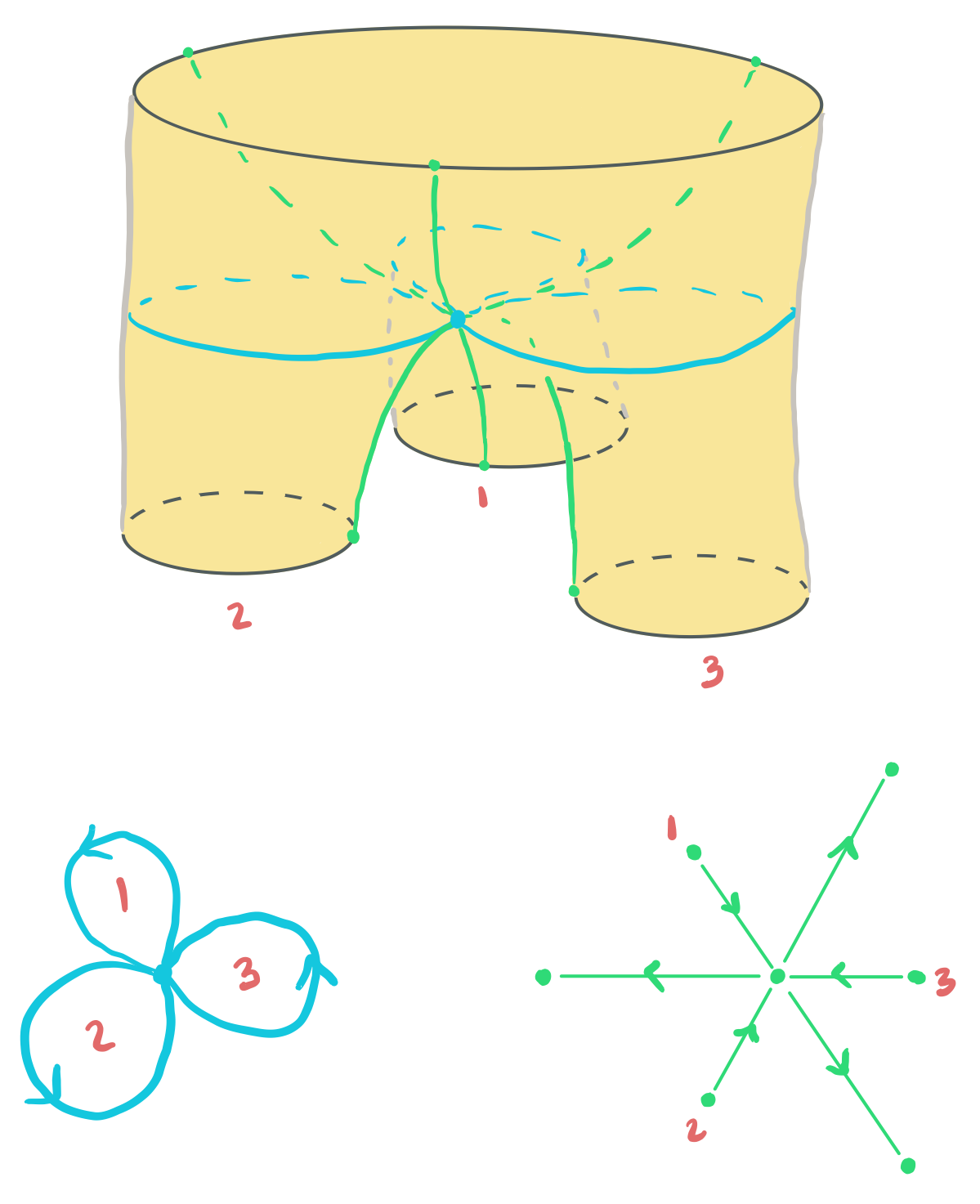}
  \caption{The branched annulus for a cubic polynomial of the form
    described in Example~\ref{ex:lvl-dir}, along with its unique
    critical level set and direction set.}
  \label{fig:lvl-dir}
\end{figure}

By Lemma~\ref{lem:crit-vert}, Example~\ref{ex:lvl-dir} also describes
the generic local picture for what level sets and direction sets look
like near the image of a critical point in $\Bb_p$. We begin with the
regular case.

\begin{lem}[Product preimages]\label{lem:product-preimages}
  Let $T\subset\I$ and $U\subset\S$ be connected subsets such that the
  product region $T\times U \subset \A$ is disjoint from $i_A(\cvl)$.
  If $U$ is a proper subset of $\S$, then the preimage
  $p_A^{-1}(T\times U)$ has $d$ connected components, each of which is
  isometric to the rectangle $T\times U$.  If $U = \S$, then there are
  positive integers $k_1,\ldots,k_m$ with $k_1 + \cdots + k_m = d$
  such that the preimage $p_A^{-1}(T\times U)$ is isometric to the
  disjoint union of annuli $(T\times \S(k_1))\sqcup \cdots \sqcup
  (T\times\S(k_m))$.
\end{lem}

\begin{proof}
  First, note that the restricted map $p_A\colon p_A^{-1}(T\times
  U)\to T\times U$ is a $d$-fold covering map and a local isometry. If
  $U$ is a proper subset of $\S$, then $T\times U$ is simply-connected
  and thus the preimage under $p_A$ must consist of $d$ isometric
  copies of $T\times U$.  If $U=\S$, then $T\times U$ is an annulus of
  circumference $2\pi$, so each connected component of the preimage is
  an annulus with circumference equal to an integer multiple of
  $2\pi$.
\end{proof}

Note that the rectangles and annuli described might be degenerate if
$T$ or $U$ (or both) is only a single point.  Next, we add in the cell
structure.  Recall that $\Ab_p = \Ib_p \times \Sb_p$.  By composing
$\pb \colon\Bb_p\to\Ab_p$ with the projection onto either factor,
there are natural cellular maps from the branched annulus to the
interval and to the circle.  The preimages of open edges under these
compositions can easily be characterized, providing a concrete
description of both regular level sets and regular direction sets.
The following two results are immediate consequences of 
Lemma~\ref{lem:product-preimages}. 

\begin{lem}[Regular level sets]\label{lem:reg-lvl-sets}
  If $T$ is an open edge in $\Ib_p$, then its preimage in $\Bb_p$ is a
  union of disjoint annuli, and each annulus is a direct metric
  product of the open edge $T$ with a finite-sheeted cover of $\Sb_p$.  
  More precisely, there exist positive integers $k_1,\ldots,k_m$ with $k_1
  + \cdots + k_m = d$ such that the preimage of $T$ is a disjoint
  union of annuli $(T\times \Sb_p(k_1))\sqcup \cdots \sqcup
  (T\times\Sb_p(k_m))$. As a consequence, for each point $t\in T$, the
  regular level set $\lvl_t$ looks like $\Sb_p(k_1) \sqcup \cdots
  \sqcup \Sb_p(k_m)$ and its metric cell structure is independent of
  the choice of $t\in T$.
\end{lem}

\begin{lem}[Regular direction sets]\label{lem:reg-dir-sets}
  If $U$ is an open edge in $\Sb_p$, then its preimage in $\Bb_p$ is a
  union of $d$ disjoint rectangles, and each rectangle is a metric
  product of $\Ib_p$ and the open edge $U$. As a consequence, for each
  point $u \in U$, the regular direction set $\dir_u$ looks like
  $\Ib_p \sqcup \cdots \sqcup \Ib_p$ and its metric cell structure is
  independent of the choice of $u\in U$.
\end{lem}

Removing a regular direction set leaves a branched cover of a
rectangle.

\begin{lem}[Removing a regular direction set]\label{lem:remove-reg-dir}
  Let $T$ be an open edge in $\Sb_p$.  When the preimage of $T \times
  \Ib_p$ under $\pb$ is removed from the branched annulus $\Bb_p$, the
  result is a contractible, connected branched cover of the rectangle
  $\Ab_p \setminus (T\times \Ib_p)$. 
\end{lem}

\begin{proof}
  By Lemma~\ref{lem:reg-lvl-sets} the inverse image of the strip
  $T\times \Ib_p$ under $\pb$ consists of $d$ disjoint copies of the
  strip, with each copy connecting a $T$-edge in a root circle to one
  of the $d$ $T$-edges in the circle at infinity.  Moreover, distinct
  strips start at distinct root circles, since each root circle has a
  unique edge sent to the copy of $T$ in $\lat_{-1}$.  The fundamental
  group of the branched annulus $\Bb_p$ is a rank~$d$ free group
  generated by loops running around the root circles, but each strip
  that is removed cuts the surface and reduces the rank of this free
  group.  The final surface is connected and contractible with a
  single boundary cycle, and it is a branched cover by
  Proposition~\ref{prop:poly-maps}. Notice that when $\Sb_p$ has only
  one edge, the `rectangle' $\Ab_p \setminus (T\times \Ib_p)$ is
  degenerate, being only a single copy of $\Ib_p$. The result still
  holds, but the `surface' degenerates into a planar tree.
\end{proof}

Critical level sets and critical direction sets are more complicated
due to the branch points.  A connected component of a
critical level set is a special type of metric graph that we call a
\emph{metric cactus}.  A metric cactus may be of the form $\Sb_p(k)$,
as in the regular case, but it can also contain branch points, as in
Figure~\ref{fig:lvl-dir}.

\begin{defn}[Metric cacti]\label{def:metric-cacti}
  A \emph{cactus diagram} is a contractible $2$-dimensional cell
  complex embedded in the plane where each edge lies in the boundary
  of a unique $2$-cell.  As a consequence of these restrictions, the
  boundary of an open $2$-cell is a (simple) cycle, disjoint open
  $2$-cells have closures that intersect in at most one point, and the
  closed $2$-cells are assembled in a ``tree-like'' fashion. In graph
  theory a \emph{cactus} is a graph where every edge belongs to a
  unique cycle, and the $1$-skeleton of a cactus diagram is an
  example.  Note that loops and multiple edges are permitted.
  Actually, the $1$-skeleton of a cactus diagram can be viewed as a
  directed graph by giving a counter-clockwise orientation to the
  boundary cycle of each $2$-cell.  Let $\Sb$ be a circle of length
  $2\pi$ with a fixed cell structure and consistently oriented edges.
  If $\Gamma$ is a branched cover of $\Sb$ and it is also the
  oriented $1$-skeleton of a cactus diagram, then we call $\Gamma$ a
  \emph{metric cactus}.  Note that every simple cycle in $\Gamma$ is a
  copy of $\Sb(k)$ for some positive integer $k$. 
\end{defn}

Cactus diagrams and cactus graphs appear in \cite{nekrashevych14}, but
it is worth noting that they are put to a different use.  Cactus
cycles in \cite{nekrashevych14} correspond to post-critical points of
$p$, whereas ours are closely connected to roots.

\begin{lem}[Critical level sets]\label{lem:critical-levels}
  Inside $\Bb_p$, every component of every level set is a metric
  cactus.  Moreover, each critical point of multiplicity $k$ labels a
  cactus vertex of valence $2(k+1)$ and it belongs to exactly $k+1$
  distinct cycles.
\end{lem}

\begin{proof}
  The topological part of the proof is Morse-theoretic.  For each $r
  \geq 0$, let $C(r) = \{z \mid \size{p(z)} \leq r\} \subset \C$. When
  $r$ is positive, we upgrade $C(r)$ from a topological subset of the
  plane to a planar cell complex by giving its boundary the cell
  structure of the corresponding critical level set in $\Bb_p$.  Note
  that every connected component of $C(r)$ is contractible since a
  bounded complementary component would lead to a contradiction of the
  Maximum Modulus Principle.  In particular, as $r$ continues to
  increase, the bounded region eventually must disappear and the final
  points to disappear will be isolated since the critical points of
  $p$ are discrete.  Such points would have a locally maximal modulus,
  thus provoking a contradiction.
  When $r=0$, $C(r)$ is the discrete set of roots and as $r$
  increases, it immediately becomes a set of $d$ disjoint closed
  topological disks.  Note that each connected component is a cactus
  diagram with only one $2$-cell.  As $r$ increases further, some of
  the growing closed disks intersect precisely when $r$ has the
  magnitude of a critical point of $p$.  When this happens the local
  picture is governed by Lemma~\ref{lem:crit-vert}. This establishes
  the valence requirements and the closed disks only overlap at
  isolated points because the critical values of $p$ are discrete.
  Thus, $C(r)$ remains a cactus diagram during this transition.  As a
  branched cover of $\Sb_p$ and the $1$-skeleton of a cactus diagram,
  the critical level set corresponding to $\partial C(r)$ is a metric
  cactus.  As $r$ increases further, annuli are attached to the
  boundary cycles of the connected components, and the space $C(r)$
  returns to being a disjoint union of closed topological disks, but
  now with strictly fewer connected components.  Proceeding in this
  way shows that every level set is a metric cactus.
  \end{proof}

\begin{figure}
  \includegraphics{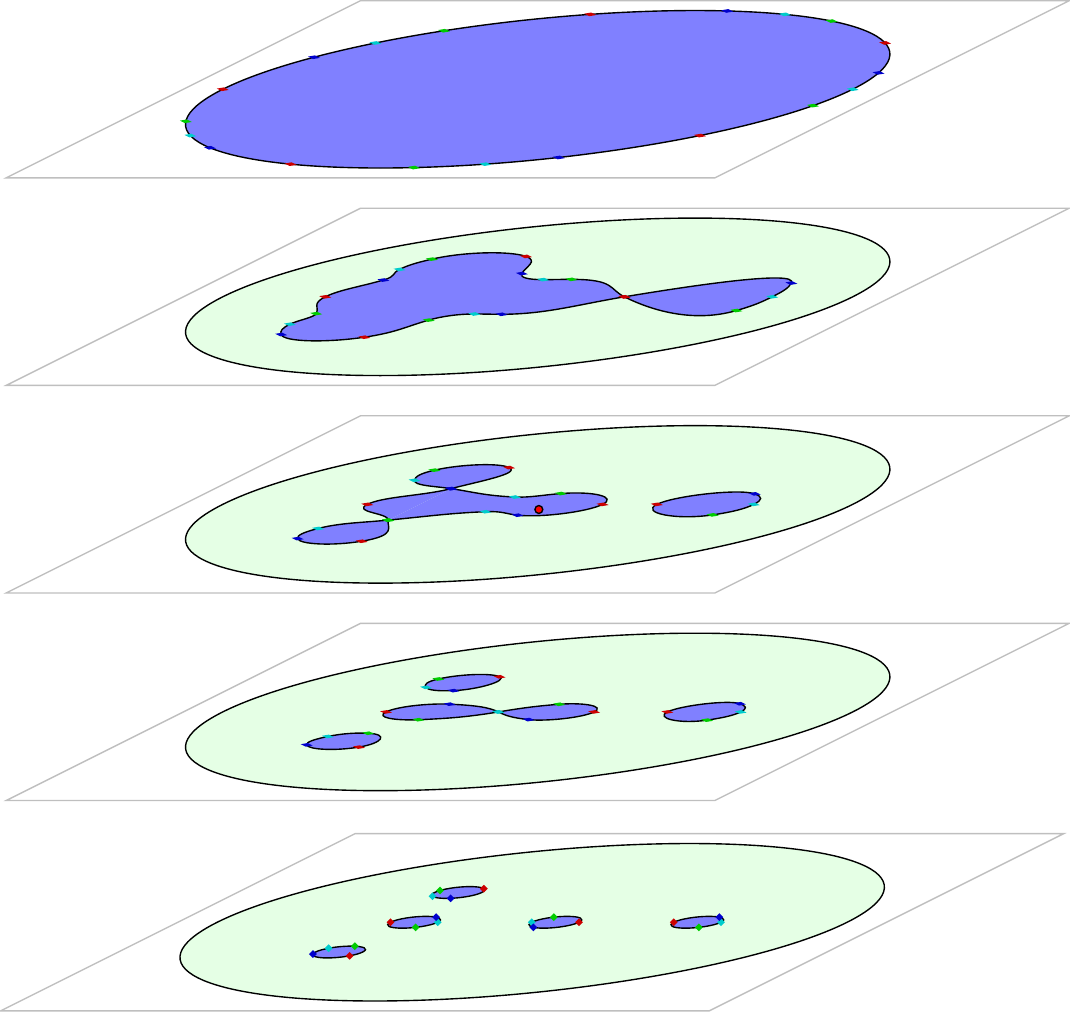}
\caption{The critical horizontal slices of the branched annulus
  $\Bb_p$ as embedded in $\D \times \I$, together with its two extreme
  level sets at the top and bottom of the solid
  cylinder.\label{fig:cactus-multipedal}}
\end{figure}

\begin{rem}[Multipedal pants and cactus diagrams]
  An alternative way to visualize the union of cactus diagrams that
  level sets bound is to use the embedding of $\Bb_p$ as a multipedal
  pair of pants inside $\D \times \I$.  The pair of pants have an
  ``inside'' and an ``outside''.  If you add disks to the portions of
  horizontal slices that are on the ``inside'', then the result is a
  union of cactus diagram.  The critical level sets corresponding and
  the corresponding cactus diagrams for our running example are shown
  in Figure~\ref{fig:cactus-multipedal}, together with the two extreme
  level sets at top and bottom the solid cylinder.
\end{rem}

A connected component of a critical direction set is a special type of
metric tree that we call a \emph{metric banyan}.  A metric banyan may
be of the form $\Ib_p$, as in the regular case, but it can also
contain branch points, as in Figure~\ref{fig:lvl-dir}.  

\begin{defn}[Banyan trees]
  Let $\Ib$ be a compact metric interval with a fixed cell structure
  and with oriented edges consistently pointing towards one of its
  endpoints.  A \emph{metric banyan tree} is a contractible branched
  cover of $\Ib$ together with a fixed embedding in the plane so that
  the incoming and outgoing edges strictly alternate in the clockwise
  cyclic order around one of its internal vertices.  Notice that every
  point in a metric banyan has a well defined \emph{height} given by
  its projection to $\Ib$.  A vertex of a metric banyan tree is called
  a source, sink or interior vertex based on its image in $\Ib$.
\end{defn}

\begin{rem}
The choice of name ``banyan tree'' is inspired by the behavior of the
real-life version, where branches grow off-shoots heading both towards
and away from the ground.  By adding the well-defined height as a
third dimension, one may notice that the lifts of the planar trees
described above mimic this behavior.
\end{rem}

\begin{lem}[Critical direction sets]\label{lem:crit-directions}
  Inside a branched annulus $\Bb_p$, every component of every
  direction set is a metric banyan. Moreover, each critical point of
  multiplicity $k$ labels an interior banyan vertex of valence
  $2(k+1)$.
\end{lem}

\begin{proof}
  By Proposition~\ref{prop:poly-maps}, every singular component of a
  direction set is a branched covering of $\Ib_p$.  The planar
  embedding, the local alternation of incoming and outgoing edges, and
  the valence requirement are all immediate, since the local picture
  near a branch point is governed by Lemma~\ref{lem:crit-vert}.  It
  only remains to show that every component is contractible and thus a
  tree.  This is clear for regular direction sets by
  Lemma~\ref{lem:reg-lvl-sets}, so we only need to consider the
  critical ones.  Let $T$ be an open edge in $\Sb_p$ and consider the
  inverse image of the strip $T\times \Ib_p$ under $\pb$.  By
  Lemma~\ref{lem:remove-reg-dir}, removing these $d$ strips results in
  a connected and contractible branched cover of a rectangle.  If we
  repeat this process for each other open edge in $\Sb_p$, then these
  additional ``cuts'' increase the number of connected components, but
  each component remains simply connected. After all such strips have
  been removed, all that remains is the disjoint union of all critical
  direction sets.  This shows that every component of every critical
  direction set is simply connected and contractible.
\end{proof}

We conclude this section with two brief remarks that will be explored
in greater depth in future articles by the authors.  The first remark
highlights the underlying similarity between banyan trees and cacti.

\begin{rem}[Banyan and cactus conversions]\label{rem:banyan-cactus-convert}
  Let $\Gamma$ be a level set of $\Bb_p$ and let $T$ be an open edge
  in $\Sb_p$. Then removing the preimages of $T\times\Ib_p$ (as
  described in Lemma~\ref{lem:remove-reg-dir}) from $\Bb_p$ removes
  $d$ open edges from $\Gamma$; one can show that the resulting metric
  graph is a metric banyan.  Conversely, if $\Gamma$ is a direction
  set of $\Bb_p$ and $T$ is an open edge in $\Sb_p$, then removing the
  preimages of $T\times\Ib_p$ turns $\Bb_p$ into a $4d$-gon that is a
  connected $d$-fold branched cover of the rectangle $\Ab_p \setminus
  (T \times \Ib_p)$ (Lemma~\ref{lem:remove-reg-dir}).  Every fourth
  side of the $4d$-gon contains a source vertex of $\Gamma$, every
  fourth side of the $4d$-gon contains a sink vertex of $\Gamma$ and
  none of these $2d$ sides are adjacent.  One can add a collection of
  $d$ isometric metric edges to $\Gamma$ connecting each source vertex
  to the ``next'' sink vertex in the counter-clockwise order.  This
  can be done simultaneously in a plane and the result is a graph
  where each component is a metric cactus.  This close connection
  between metric banyans and metric cacti is less surprising once one
  uses Lemma~\ref{lem:remove-reg-dir} to restrict a branched annulus
  to a branched rectangle. There is an obvious symmetry between the
  horizontal and vertical foliations of the rectangle and this
  produces a symmetry between metric banyans and restricted portions
  of metric cacti.
\end{rem}

The second remark focuses on realizability.

\begin{rem}[Banyan and cactus realizations]\label{rem:banyan-cactus-realize}
  One can prove stronger versions of Lemma~\ref{lem:critical-levels}
  and Lemma~\ref{lem:crit-directions} that include an assertion about
  the converse direction.  Concretely, after adding mild and obvious
  restrictions, every metric cactus and every metric banyan arises as
  a component of a level set or a direction set for some complex
  polynomial $p$.
\end{rem}

\section{Partitions}\label{sec:partitions}

The branched annulus of a complex polynomial contains a wealth of
combinatorial (and geometric) data encoded in its level sets and
direction sets.  In this section we extract some of this combinatorial
data and highlight an intrinsic connection between the geometry of
complex polynomials and the combinatorics of noncrossing
partitions. We begin with partitions.

\begin{defn}[Partitions]
  Let $A$ be a set of size $d$.  A \emph{partition} of $A$ is a
  collection of pairwise disjoint subsets (called \emph{blocks}) whose
  union is $A$.  Let $\Pi_A$ denote the set of all such
  partitions. Given partitions $\lambda, \mu \in \Pi_A$, $\lambda$ is
  a \emph{refinement} of $\mu$ if each block of $\lambda$ is contained
  in some block of $\mu$.  When this occurs we write $\lambda \leq
  \mu$, or $\lambda < \mu$ when $\lambda$ and $\mu$ are not
  equal. This defines a partial order which makes $\Pi_A$ into a
  \emph{lattice} since meets and joins are well-defined.  The unique
  minimum partition is the \emph{discrete partition} where every block
  is a singleton, and the unique maximum partition is the
  \emph{trivial partition} where there is only one block.  A
  \emph{chain of partitions} is a sequence $\lambda_1 < \cdots <
  \lambda_k$.
\end{defn}

The focus here is on $\Pi_\rts$, the partition lattice of the roots of
$p$.

\begin{defn}[Partitions from level sets]\label{def:lat-partition}
  As described in the proof of Lemma~\ref{lem:critical-levels}, for
  every $t \in \I^\inter$, the level set $\lvl_t$ of $p$ can be viewed
  as the boundary of a disjoint union of cactus diagrams in $\C$ and
  every root is contained in the interior of one of these diagrams.
  We define a partition $\lambda_t$ by placing two roots in the same
  block if and only if they are contained in the same connected
  component, i.e. the same cactus diagram.  Note that as $t$
  increases, the cactus diagrams grow and merge but they do not split,
  and, as a consequence $\lambda_t$ can only increase in the partial
  order on $\Pi_\rts$.
\end{defn}

\begin{rem}[Regular level sets as separating multicurves]
  A \emph{multicurve} is a disjoint union of simple closed curves on a
  given (connected) surface, and a multicurve is \emph{separating} if
  its complement is disconnected. Since each latitude $\lat_t$ in the
  open annulus $\A^\inter$ disconnects the (closed) annulus $\A$, its
  preimage $\lvl_t$ separates $\B_p$.  Moreover, when $t$ is regular,
  $\lvl_t$ is a separating multicurve. The partition $\lambda_t$ can
  then be viewed as having blocks where two roots belong to the same
  block if and only if the corresponding root circles belong to the
  same complementary component of $\lvl_t$.
\end{rem}

\begin{rem}[Partitions and $\Ib_p$]
  For every point $t \in \Ib_p$ there is a partition $\lambda_t \in
  \Pi_\rts$.  As is usual in Morse theory, the cactus diagrams remain
  homotopic so long as $t$ does not increase through a critical value.
  Thus every point in an open edge of $\Ib_p$ is assigned the same
  partition, and this partition agrees with the partition assigned to
  its lower endpoint.  Moreover, the partition assigned to the first
  open edge (with $-1$ as an endpoint) is the discrete partition and
  the partition assigned to the last open edge (with $1$ as an
  endpoint) is the trivial partition.  On the other hand, as $t$
  approaches an interior vertex of $\Ib_p$ from below, there are
  distinct components that are merging together, which means that the
  partition assigned to this interior vertex is strictly higher
  inside $\Pi_\rts$ compared to the partition assigned to the open
  edge directly below it.  As a consequence, the partitions assigned
  to the open edges of $\Ib_p$ are distinct representatives of every
  partition assigned to a level set, and they naturally form a chain
  in $\Pi_\rts$.
\end{rem}

We give two examples.

\begin{example}\label{ex:extreme-partitions}
  Let $p(z) = a(z-b)^d+c$ be as in Example~\ref{ex:lvl-dir}, and let
  $t$ be the height of $i_A(c)$, its unique critical value.  There are
  only two open edges in $\Ib_p$. The partition assigned is discrete
  on the half-open interval $[-1,t)$ and trivial on the closed
    interval $[t,1]$.  Thus the chain of partitions is simply:
    $\textrm{Discrete} < \textrm{Trivial}$.
\end{example}

\begin{figure}
    \centering
    \includegraphics[scale=0.6]{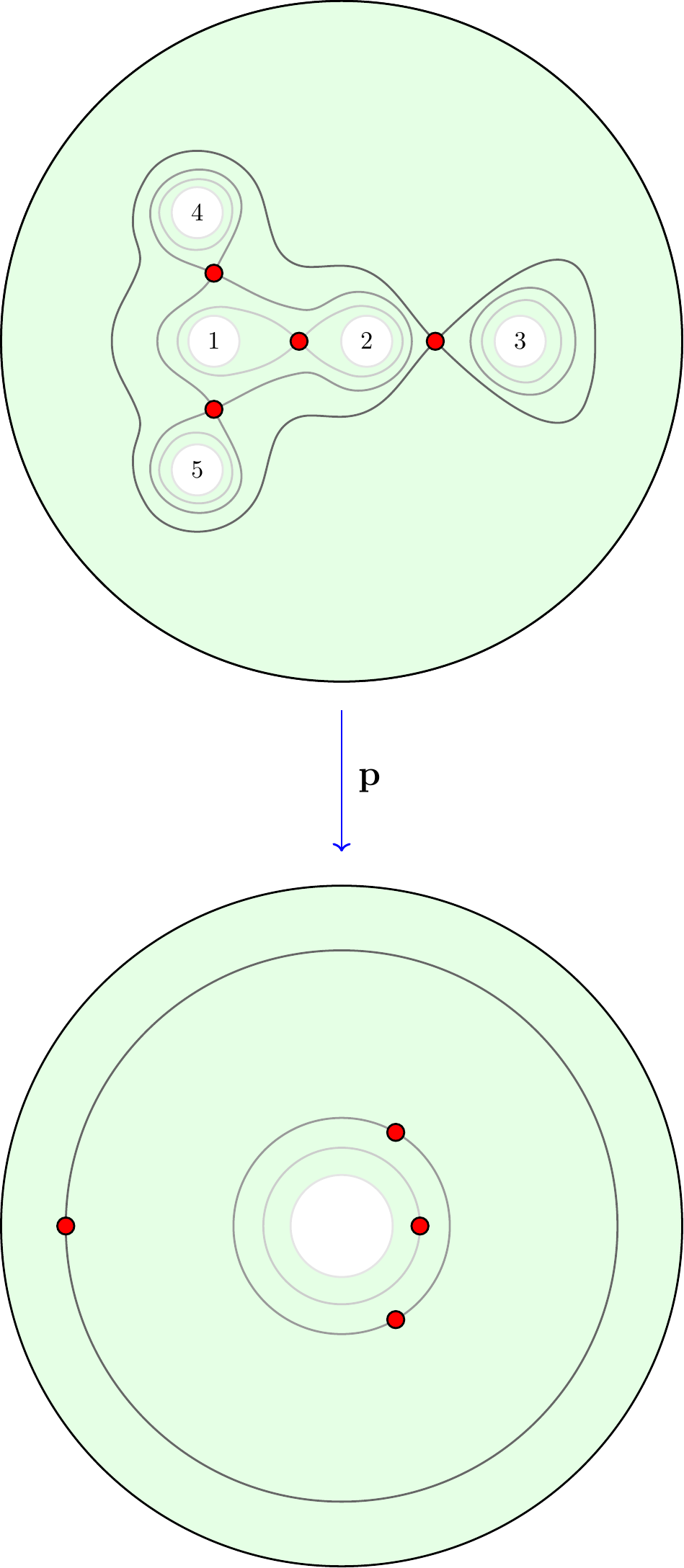}
    \caption{The union of all critical level sets in the branched
      annulus $\Bb_p$.  The regular level sets between them determine
      a chain in $\Pi_\rts$.\label{fig:lat-pullback}}
\end{figure}

\begin{example}\label{ex:partition-chain}
  The critical level sets for our standard running example are shown
  in Figure~\ref{fig:lat-pullback}.  As should be clear from the
  figure, the associated chain in $\Pi_\rts$ is:
  \[ \textrm{Discrete} < \{\{a_1,a_2\},\{a_3\},\{a_4\},\{a_5\}\} <
  \{\{a_1,a_2,a_4,a_5\},\{a_3\}\} < \textrm{Trivial}.\]
\end{example}

\begin{rem}[Big lemniscate configurations]
  The chain of partitions produced by the regular level sets of $p$ is
  a natural combinatorial object, although we are unaware of any
  explicit references to it in the literature. It is worth noting that
  this chain is related to what Catanese and Paluszny refer to as a
  ``big lemniscate configuration'' in \cite{catanese-paluszny},
  although they restrict their consideration to the lemniscate-generic
  case, where the combinatorics are simpler.  A lemniscate-generic
  polynomial of degree $d$ has $d-1$ distinct critical values, all of
  multiplicity one, and these critical values all have distinct
  moduli.  The cell structure of $\Ib_p$, in this case, has $d$ open
  edges and the associated chain in $\Pi_\rts$ is a maximal chain.
\end{rem}

With a little more work, the chain of partitions determined by the
regular level sets of $\Bb_p$ can be converted into a chain of
noncrossing partitions.  

\begin{defn}[Noncrossing partitions]\label{def:noncrossing}
  Let $A$ be the vertices of a convex $d$-gon in the plane. A
  partition $\lambda \in \Pi_A$ is \emph{noncrossing} if for every
  pair of distinct blocks in $\lambda$, the convex hull of the
  vertices in one block is disjoint from the convex hull of the
  vertices in the other.  The collection of noncrossing partitions
  inside $\Pi_A$ form an induced subposet.  We write $\NC_A$ for this
  subposet, where $A$ is now viewed as a set with a fixed cyclic
  ordering.  If $A$ is any finite set of size $d$ and we fix a cyclic
  ordering of its elements, then we can use this cyclic ordering to
  label the vertices of a $d$-gon and then create an induced
  noncrossing subposet of $\Pi_A$.
\end{defn}

\begin{lem}[Open edges and cyclic orderings]\label{lem:open-edge-cyclic}
  If $u\in \Sb_p$ is not a vertex, then $\dir_u$ is a regular
  direction set, and $\dir_u$ determines a natural cyclic ordering of
  the roots of $p$.  In addition, points in the same open edge of
  $\Sb_p$ determine the same cyclic ordering of $\rts$.
\end{lem}

\begin{proof}
   Let $U$ be the open edge of $\Sb_p$ that contains $u$.  The
   preimage of the strip $U\times\Ib_p \subset \A$ consists of $d$
   disjoint rectangular strips in $\Bb_p$.  Each strip connects a
   unique copy of $U$ on a root circle to one of $d$ copies of $U$ on
   the circle at infinity. The bijection from roots to root circle
   thus extends to copies of $U$ in the circle at infinity.  And note
   that the same bijection is produced if we use only use the regular
   direction set $\dir_u$ for any $u\in U$.  Reading counter-clockwise
   around the circle at infinity thus produces a cyclic ordering of
   the set $\rts$ of roots.
\end{proof}

We illustrate the cyclic ordering of the roots using our running
example.

\begin{example}[Cyclic orderings]\label{ex:cyclic-ord}
  In our running example, there are four open edges in $\Sb_p$ and
  thus four cyclic orderings of the roots.  For readability, we
  label the five roots $a$, $b$, $c$, $d$ and $e$ instead of $a_1$,
  $a_2$, $a_3$, $a_4$ and $a_5$, respectively.  Starting at the
  positive $x$-axis and proceeding counter-clockwise around the
  boundary, these four cyclic orderings are: $(c,a,d,e,b)$,
  $(c,d,a,e,b)$, $(b,d,a,e,c)$, and $(b,d,e,a,c)$.  See
  Figure~\ref{fig:long-pullback}.  In the electronic version of this
  article, where the colors are visible, the first cyclic ordering is
  determined by the regular direction sets sandwiched between the
  critical direction sets that are aqua and dark blue, respectively.
  The second is between dark blue and red, the third is between red
  and green and the fourth is between green and aqua.
\end{example}

\begin{figure}
    \centering
    \includegraphics[scale=0.6]{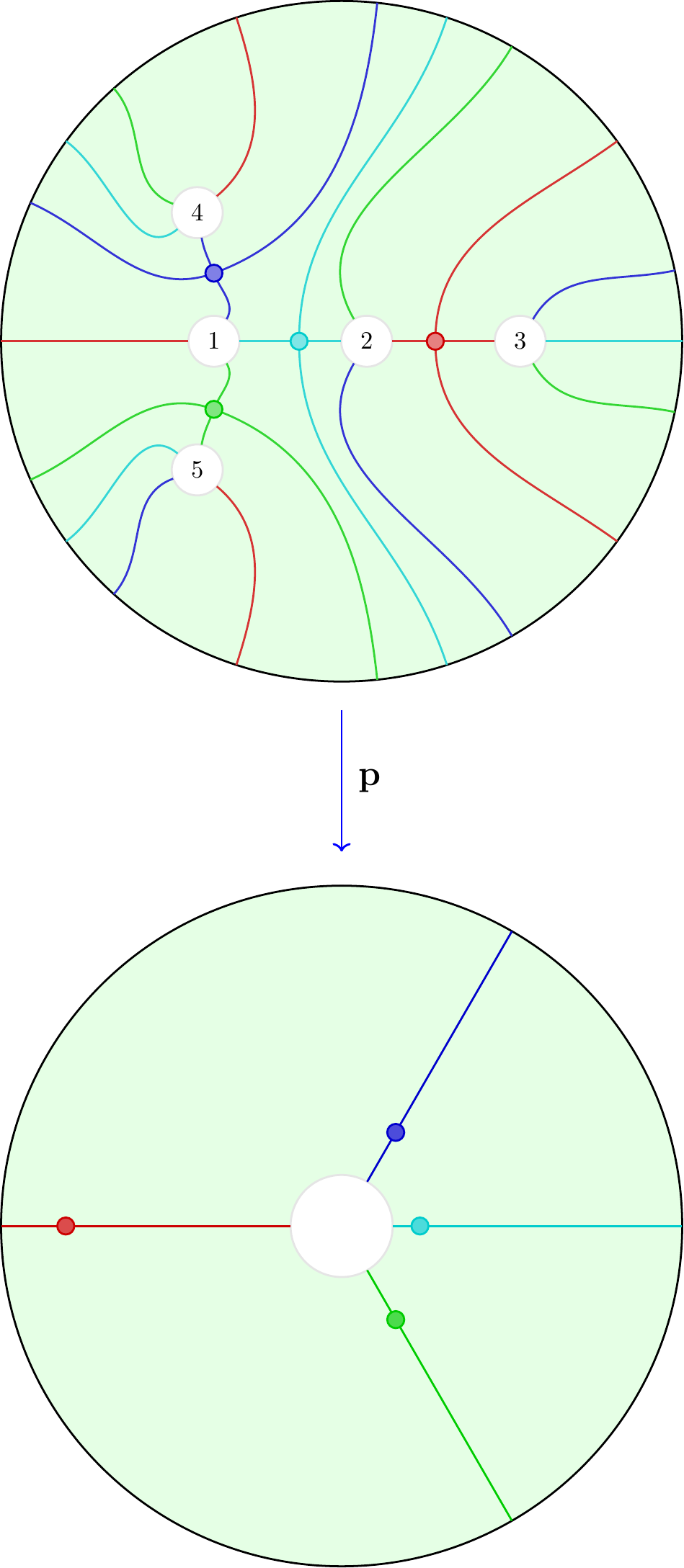}
    \caption{The union of all critical direction sets in the branched
      annulus $\Bb_p$.  The regular level sets between them determine
      four different cyclic orderings of the set
      $\rts$.\label{fig:long-pullback}}
\end{figure}

By combining the combinatorial information coming from a regular
direction set with that coming from a regular level set we can 
see that all of these root partitions are noncrossing.

\begin{prop}[Level sets and noncrossing partitions]\label{prop:nc-level-parts}
  The partition of the roots determined by a regular level set is
  noncrossing with respect to the cyclic ordering of the roots
  determined by a regular direction set.
\end{prop}

\begin{proof}
  Fix $u\in \Sb_p$ and $t\in \Ib_p$ so that neither one is a vertex.
  By Lemma~\ref{lem:open-edge-cyclic}, $u$ determines a corresponding
  cyclic ordering of the set $\rts$ and by
  Definition~\ref{def:noncrossing} this determines a subposet $\NC_\rts$
  inside $\Pi_\rts$.  Let $C$ be the union of the latitude circle
  $\lng_t$ and the longitude line $\lat_u$, and let $B = p_A^{-1}(C)$.
  A connected component of $B$ contains both the $1$-skeleton of a
  single cactus diagram, and the lines that connect some of the root
  circles to preimages of $u$ in the circle at infinity.  These
  connected components show how to view a block of the partition
  $\lambda_t$ as a collection of preimages of $u$ in the circle at
  infinity.  Since all of the components of $B$ are drawn in the disk
  without crossing, the convex hulls of the collections of points in
  the circle at infinity also do not cross.  In particular,
  $\lambda_t$ is in the subposet $\NC_\rts$.
\end{proof}

The following corollary is now immediate.

\begin{cor}[Chains of noncrossing partitions]\label{cor:nc-chains}
  The chain of partitions determined by the regular level sets is
  noncrossing with respect to the cyclic ordering of the roots
  determined by any regular direction set.
\end{cor}

\begin{rem}
It is worth noting that connections between complex polynomials and
noncrossing partitions have appeared elsewhere in the literature
(e.g. \cite{martin07} and \cite{savitt09}), although in slightly
different manners. We are unaware of any previous results in the
literature that are similar to Proposition~\ref{prop:nc-level-parts}
and Corollary~\ref{cor:nc-chains}.
\end{rem}

\section{Factorizations}\label{sec:factorizations}

We now shift our attention from the chain of noncrossing partitions
determined by the collection of regular level sets to the
factorization of a $d$-cycle determined by the collection of regular
direction sets.  As a first step we note that the cyclic orderings
described in Lemma~\ref{lem:open-edge-cyclic} can be converted to
linear orderings by focusing on a particular edge in the circle at
infinity.

\begin{defn}[Linear orderings]\label{def:linear-orderings}
  Let $U'$ be an open edge in the circle at infinity in $\Bb_p$ and
  let $U$ be the open edge $\Sb_p$ to which it projects.  First label
  the preimages of $U$ in the circle at infinity using the bijection
  from Lemma~\ref{lem:open-edge-cyclic}.  Starting at $U'$ and
  proceeding in a counter-clockwise fashion, we record the linear
  order in which the root labels occur.  We use square brackets,
  rather than parentheses, to indicate that this is a linear ordering
  rather than a cyclic ordering.  
\end{defn}

The various linear orderings of the roots determined by the edges in
the circle at infinity, can be used to create a sequence of
permutations whose product is an $d$-cycle.  We begin by showing how
that works in our running example.

\begin{example}\label{ex:factorization}
  The circle at infinity in our running example has $20$ open edges.
  Using the same conventions as in Example~\ref{ex:cyclic-ord}, we
  have, for the edge in the domain in the first quadrant with one
  endpoint on the positive real axis, the linear order is
  $[c,a,d,e,b]$.  Then next several linear orders, proceeding in a
  counter-clockwise fashion, are $[c,d,a,e,b]$, $[b,d,a,e,c]$,
  $[b,d,e,a,c]$ and $[a,d,e,b,c]$.  There are $15$ more linear orders.
  Of course, the fifth linear order is just a cyclic permutation of
  the first since they correspond to the same set of cyclically
  ordered edges.  Adjacent linear orders in this list give a
  description of a permutation of the set $\{1,2,3,4,5\}$ of positions
  in two line notation.  The permutation from the first to the second
  linear order is $(2,3)$ since the $a$ and $d$ entries in positions
  $2$ and $3$ are swapped.  The next permutation is $(1,5)$, swapping
  $c$ and $b$, then $(3,4)$ swapping $a$ and $e$, and finally $(1,4)$
  swapping $a$ and $b$.  Note that the product (with the appropriate 
  choice of conventions for multiplying permutations) is 
  $(1,4)(3,4)(1,5)(2,3) = (1,2,3,4,5)$ as expected.
\end{example}

\begin{rem}[Banyans and permutations]
  Given the fact that critical direction sets separate adjacent linear
  and cyclic orderings, these permutations can actually be read off of
  the corresponding banyan trees, specifically from the banyan trees
  that include branch points.  For example, the dark blue critical
  direction set has only one branched component, the one with source
  vertices on the $a_1=a$ and $a_4=d$ root circles.  As a consequence,
  the difference between the first linear order $[c,a,d,e,b]$ and the
  second $[c,d,a,e,b]$ involves a swapping of the roots $a$ and $d$.
\end{rem}

In order to describe the permutations these trees produce, we first
introduce a new noncrossing object, known in the literature as a
primitive $d$-major.

\begin{defn}[Real noncrossing partitions]
  Let $d$ be a positive integer, let $\S(d) \subset \C$ be the circle
  of radius $d$ (and circumference $2\pi d$), and let $\S(d) \to \S$
  be the natural covering map sending $z \mapsto (z/d)^d$.  For each
  point $u \in \S$, its $d$ preimages are equally spaced around
  $\S(d)$ and they form the vertex set of regular $d$-gon.  Thus we
  can define a noncrossing partition on the preimages of $u$.  A
  \emph{real noncrossing partition $\lambda$} a choice of a
  noncrossing partition for each point $u \in \S$ subject to a
  compatibility condition, which requires that the convex hull of
  every block of every selected noncrossing partition can be
  simultaneously drawn inside the disk that $\S(d)$ bounds, while
  remaining pairwise disjoint.  One consequence of this compatibility
  condition is that all but finitely many of the selected noncrossing
  partitions must be the discrete partition.  Note that the set of all
  real noncrossing partitions comes with a natural partial order.  One
  real noncrossing partition is less than or equal to another if for
  each $u\in \S$ the noncrossing partition selected by the first is
  less than or equal to the noncrossing partition selected by the
  second in the noncrossing partition lattice based on the preimages
  of $u$.
\end{defn}

\begin{rem}
  Real noncrossing partitions have already appeared in the literature
  with a different definition and under a different name.  In
  \cite{thurston19} they are called \emph{primitive $d$-majors}.  If
  each of the non-trivial convex hulls in the noncrossing partitions
  of a real noncrossing partition are shrunk to a point in a planar
  fashion, the result is a cactus diagram with a metric cactus as its
  boundary.  In fact, real noncrossing partition are in natural
  bijection with metric cacti that are branched $d$-fold covers of
  $\S$.
\end{rem}

\begin{defn}[Banyans and real noncrossing partitions]
  For each $u\in \S$, consider the $d$ points in $\dir_u$ which lie on
  the circle at infinity. Define a partition of these points so that
  two points belong to the same block if and only if they lie on the
  same connected component of $\dir_u$.  For regular direction sets,
  this is the trivial partition.  The proof that these partitions are
  noncrossing is essentially the same as in the proof of
  Proposition~\ref{prop:nc-level-parts}.  Since all of the components
  of $\dir_u$ are drawn in the disk without crossing, the convex hulls
  of the corresponding collections of points in the circle at infinity
  also do not cross.  This also shows that the noncrossing partitions
  associated to different points $u$ are compatible, so that the full
  banyan foliation of $\B_p$ determines a real noncrossing partition
  associated to the polynomial $p$.
\end{defn}

The factorization of the $d$-cycle illustrated in
Example~\ref{ex:factorization} is closely related to the real
noncrossing partition of the banyan foliation.  The following result
precisely describes the relationship, even though its proof is merely
sketched.

\begin{prop}[Real noncrossing partitions and noncrossing hypertrees]
  Let $\lambda$ be the real noncrossing partition associated to the
  foliation of $\B_p$ by banyan trees. If $u\in \Sb_p$ is a point that
  is not a vertex, then $\lambda$ corresponds to a noncrossing
  hypertree and a factorization of a $d$-cycle.
\end{prop}

\begin{proof}[Proof sketch]
  Since $u$ in not a vertex in $\Sb_p$, the noncrossing partition
  associated to $u$ is discrete, and the preimages of $u$ can be
  removed from the circle at infinity without removing a vertex of a
  non-trivial block of $\lambda$.  The remaining portions of the
  circle at infinity consist of $d$ open metric arcs of length
  $2\pi$.  Cyclically label these
  arcs $1,2,\ldots,d$.  We use these numbers to label the vertices of
  each non-trivial block in $\lambda$.  If we collapse the open arcs
  to points while maintaining planarity, the non-trivial blocks of
  $\lambda$ will overlap on vertices and become a noncrossing
  hypertree.  The non-trivial blocks of $\lambda$ can also be turned
  into cyclic permutations, and the product of these permutations in
  the appropriate order produces the $d$-cycle.  See 
  \cite{mccammond17} for details.
\end{proof}

This completes the proof of Theorem~\ref{thma:combinatorics}.
We conclude the section with two remarks on interpreting and
using these results.

\begin{rem}[Different noncrossing partitions]
  The factorization of a $d$-cycle produced by a noncrossing hypertree
  is a minimum reflection length factorization of a $d$-cycle.  There
  is a well-known connection between such factorizations and chains in
  the noncrossing partition lattice $\NC_{[d]}$ and this is another
  realization of this connection.  We should caution, however, that
  the chain in the noncrossing partition lattice determined the
  regular level sets (and an edge in $\Sb_p$) is \emph{not} related to
  the noncrossing partition lattice corresponding to a factorization
  of a $d$-cycle determined by the linear orders coming from a
  consecutive sequence of edges in the circle at infinity.  This is
  easiest to see in the completely generic case where every critical
  value has multiplicity one and all of their moduli and all of their
  arguments are distinct. The chain in the partition lattice is
  determined by the linear ordering of the critical values by
  \emph{latitude}, whereas the factorization of the $d$-cycle is
  determined, in part, by the cyclic ordering of the critical values
  by \emph{longitude}.  In a future article, where we consider
  continuously varying a polynomial, it will be made clear how one of
  these two can change while the other remains fixed.
\end{rem}

\begin{rem}[Dual braid complex]
  In a future article, we will give the space of all real noncrossing
  partitions a natural topology and cell structure, and identify this
  space with a version of the \emph{dual braid complex} - see
  \cite{brady01}, \cite{bradymccammond10}, and \cite{dmw20} for
  details on this complex. We will use this identification to
  prove that the complexified hyperplane complement of the braid
  arrangement deformation retracts to the pure version of the dual
  braid complex.  The tools developed in this article, and in our
  earlier work, will be crucial to our proofs.
\end{rem}

\section{Monodromy}\label{sec:monodromy}

In this final section we comment on the ways in which the monodromy
action can be read off from the structure of the branched annulus of
$p$.  Our discussion here will be somewhat brief, since similar
material was detailed by Elias Wegert in a recent article in
the Notices of the AMS \cite{wegert20}.

\begin{defn}[Monodromy]
  Given a covering map $f\colon Y \to X$ and a point $x\in X$, each
  oriented loop based at $x$ lifts to a collection of directed paths
  in $Y$, and each path connects one preimage of $x$ to another
  preimage of $x$, possibly the same one. Thus, each loop can used to
  determine a permutation of the points in the set $A = f^{-1}(x)$,
  these permutations are well-defined up to homotopies of paths, and
  they compose as expected.  The result is a group homomorphism from
  $\pi_1(X,x)$ to $\sym_{A}$.  This is the \emph{monodromy action} of
  the fundamental group of $X$ on the preimages of $x$.
\end{defn}

\begin{rem}[Polynomial monodromy]
  In the case of a polynomial $p\colon\C\to\C$ with $d$ distinct
  roots, we know that $p$ is a discretely branched covering map; the
  restriction of $p$ obtained by removing the critical values and
  their preimages is a covering map, and thus it has a corresponding
  monodromy action. Since $p$ has distinct roots, $0$ is not a
  critical value, and we may choose $0$ to be the basepoint in the
  image and the fundamental group $\pi_1(\C_\cvl,0)$, is a free group
  whose rank is equal to the number of critical values. 
\end{rem}

The monodromy action is encoded in the cell structure of the branched
annulus.

\begin{rem}[Monodromy and branched annuli]
  The lower boundary circle of the annulus $\A$ corresponds to $0$ in
  a precise sense, so we can replace directed loops based at $0$ with
  directed paths that start and end on $\lat_{-1}$.  The lifts of such
  paths will start and end at root circles and thus define a
  permutation of the roots.  Since every branch point of the cellular
  map $\pb \colon \Bb_p \to \Ab_p$ is a vertex, it is sufficient to
  choose paths which are suitably generic and transverse to the cell
  structure.  These generic paths avoid the $0$-skeleton of $\Ab_p$
  and lift to paths that avoid the $0$-skeleton of $\Bb_p$.  Every
  element of the free group $\pi_1(\C_\cvl,0)$ has such a transverse
  representative and the combinatorial nature of the cell structures
  involved make the permutation determined by the lifts easy to
  determine.
\end{rem}

We illustrate this idea with two simple examples.

\begin{example}
  In our running example, consider the path in $\C$ that starts at the
  origin, increases along the positive real axis, circles the unique
  critical value on this ray in a counter-clockwise fashion 
  and then returns to the origin along the real axis.  In $\Ab_p$ 
  there is an equivalent path that proceeds up the corresponding longitude,
  circles around the critical value and then returns to $\lat_{-1}$
  along this longitude.  This path in $\Ab_p$ lifts to a set of $5$
  paths in $\Bb_p$.  The path that start on the $a_3$ root circle,
  returns to the $a_3$ root circle, as do the paths that start on
  $a_4$ root circle and the $a_5$ root circle. The path that starts on
  the $a_1$ root circle ends on the $a_2$ root circle and vice versa.
  This is because of the branch point between them.  Thus, the
  corresponding permutation is $(a_1,a_2)$.
\end{example}

And finally, we conclude with comment on a connection between the
monodromy and the structure of the banyan trees in $\Bb_p$.

\begin{example}
  Let $u$ be a vertex in $\Sb$ and consider the path in $\Ab_p$ that
  starts just to the right of the copy of $u$ in $\lat_{-1}$, travels
  straight up until its height exceeds that of all the critical values
  on $\lng_u$, crosses over to the righthand side of $\lng_u$ and then
  returns straight down to the lefthand side of $u$ in $\lat_{-1}$.
  The permutation determined by this path is the same as the
  noncrossing permutation that corresponds to the noncrossing
  partition associated with the critical direction set $\dir_u$ as
  part of the real noncrossing partition associated to the polynomial
  $p$.  In particular, this permutation is completely determined by
  the structure of the non-trivial banyan trees in the critical
  direction set $\dir_u$.  Any other directed path in $\Ab_p$ that
  stays close to the critical longitude $\lng_u$ also determines a
  permutation that can be read off of the structure of the metric
  banyan trees comprising the critical direction set $\dir_u$.
\end{example}

\bibliographystyle{amsalpha}
\bibliography{paper-1_arxiv}

\end{document}